\documentclass[12pt,a4paper]{article}

\usepackage{amsmath,amsfonts, amssymb, amsthm, graphicx}
\DeclareMathOperator*{\argmin}{arg\,min}

\begin{document}
\bibliographystyle{abbrv}

\title{Non-parametric change-point detection using string matching algorithms}
\author{Oliver Johnson\thanks{School of Mathematics, University of Bristol,
University Walk, Bristol, BS8 1TW, UK} \thanks{Corresponding author. Email
{\tt maotj@bristol.ac.uk}} \and Dino Sejdinovic$^{*}$
  \and James Cruise\thanks{The Department of Actuarial Mathematics and Statistics,
and the Maxwell Institute for Mathematical Sciences, Heriot-Watt University
Edinburgh Campus, Edinburgh, Scotland, EH14 4AS.}
\and Ayalvadi Ganesh$^{*}$ \and Robert Piechocki\thanks{Centre for Communications Research,
University of Bristol, Merchant Venturers Building,
Woodland Road, Bristol
BS8 1UB, UK}  }
\date{\today}
\maketitle

\newtheorem{theorem}{Theorem}[section]
\newtheorem{lemma}[theorem]{Lemma}
\newtheorem{proposition}[theorem]{Proposition}
\newtheorem{corollary}[theorem]{Corollary}
\newtheorem{conjecture}[theorem]{Conjecture}
\newtheorem{definition}[theorem]{Definition}
\newtheorem{example}[theorem]{Example}
\newtheorem{condition}{Condition}
\newtheorem{main}{Theorem}
\newtheorem{remark}[theorem]{Remark}
\hfuzz30pt

\newcommand{\var}{{\rm{Var\;}}}
\newcommand{\cov}{{\rm{Cov\;}}}
\newcommand{\tends}{\rightarrow \infty}
\newcommand{\tz}{\rightarrow 0}
\newcommand{\ep}{{\mathbb {E}}}
\newcommand{\pr}{{\mathbb {P}}}
\newcommand{\co}{{\mathbb {C}}}
\newcommand{\re}{{\mathbb {R}}}
\newcommand{\zz}{{\mathcal{Z}}}
\newcommand{\vc}[1]{\mathbf{#1}}

\newcommand{\I}{\mathbb {I}}
\newcommand{\blah}[1]{}
\newcommand{\TT}{{\mathcal{T}}}
\newcommand{\SSC}{{\mathcal{S}}}
\newcommand{\alphb}{{\mathcal{A}}}
\newcommand{\diy}{\displaystyle}
\newcommand{\elly}[2]{\ell^{(#1)}_{#2}}

\newcommand{\clr}[1]{C_{LR}(#1)}
\newcommand{\clra}[2]{C_{LR}^{(#2)}(#1)}
\newcommand{\crl}[1]{C_{RL}(#1)}
\newcommand{\crla}[2]{C_{RL}^{(#2)}(#1)}

\newcommand{\muc}[2]{\mu^{(#2)}_{#1}}
\newcommand{\bino}[2]{{\rm Bin} \left( #1, #2 \right)}
\newcommand{\pois}[1]{{\rm Po}\left(#1 \right)}
\newcommand{\bern}[1]{{\rm Bern} \left(#1 \right)}
\newcommand{\ZZZ}[2]{Z_{#1}(#2)}
\newcommand{\wgn}{\widehat{\gamma}}
\newcommand{\convd}{\stackrel{{\mathcal{D}}}{\longrightarrow}}
\newcommand{\thinning}[1]{\left( #1 \right) \circ}
\newcommand{\mean}[2]{d_{#1,#2}}
\newcommand{\meanone}[1]{d_{#1}}
\newcommand{\meanmin}[1]{d_{#1}^{\min}}

\newcommand{\NN}[1]{N_{\vc{#1}}}
\newcommand{\ONN}[1]{\overline{N}_{\vc{#1}}}
\newcommand{\clrloc}[2]{C_{LR,\vc{#2}}(#1)}
\newcommand{\plrloc}[2]{\psi_{LR,\vc{#2}}(#1)}
\newcommand{\grma}{Graph Model A}
\newcommand{\grmb}{Graph Model B}

\begin{abstract}
Given the output of a data source taking values in a finite alphabet,
we wish to detect change-points, that is times
when the statistical properties of the source
change. Motivated by 
ideas of match lengths in information theory, 
we introduce a novel non-parametric estimator which we call
CRECHE  (CRossings Enumeration CHange Estimator). 
We present simulation evidence that this estimator performs well, both 
for simulated sources and for
real data formed by concatenating text sources.
For example, we show that we can accurately detect the point
at which a source changes from a Markov chain to an IID source
with the same stationary distribution. Our estimator requires 
no assumptions about the form of the source distribution,
and avoids the need to estimate its probabilities.
Further, we establish consistency of the CRECHE estimator 
under a related toy model,
by establishing a fluid limit and using martingale arguments.
\end{abstract}

\newpage
\section{Introduction and notation}

Suppose we are
given the output of a data source, in the form of a string $x$ of
 $n$ symbols drawn from a
finite alphabet $\alphb$, but have no knowledge of the source's statistical
properties.
It is a well-studied problem to consider 
 whether the source is stationary  or, if it is
piecewise stationary, to estimate the change-points --
that is, positions at which  the source model changes. 
In Section \ref{sec:cplit}, we review existing approaches to the change-point
detection problem and describe some applications.

This paper offers a new universal non-parametric
perspective, motivated by ideas from information theory. 
Specifically, 
a substantial existing literature considers 
so-called `match lengths'. That is, as described in
Definition \ref{def:matchlength}, for
each point $i$ we can define the match length $L_i^n$ to be the length of
the shortest substring starting at $i$ which does not occur elsewhere in 
the string. For a wide class of processes, consistent
entropy estimators can be constructed from the match lengths,
as described in Section \ref{sec:entropy}, see for example  \cite[Theorem 1]{shields}. 

Our approach is motivated by the idea of considering match positions $T_i^n$,
 chosen uniformly at random from the places where a substring
of maximal match length occurs. We consider creating a directed
graph where position $i$ is linked to $T_i^n$ defined in this way.
We refer to this as \grma --
 see Definition \ref{def:matchpos}
for a formal definition.

 Heuristically, in a model with no change-points
we believe that the $T_i^n$ will be approximately uniformly distributed,
and in a model with change-points the $T_i^n$ will tend to lie in the same
region as $i$. 
We therefore define the crossings functions $\clr{j}$
and $\crl{j}$ as follows:
\begin{definition} \label{def:crossings}
For any directed graph formed by linking $i$ to $T_i^n$, 
given a putative change-point $0 \leq j \leq n-1$ 
we write
\begin{eqnarray} \clr{j} & = & \# \{ k: k < j \leq T_k^n \} 
\mbox{ for the number of left--right crossings
of $j$,} \\
\crl{j} & = & \# \{ k: T_k^n < j \leq k \} 
\mbox{ for the number of right--left crossings of $j$.}
\end{eqnarray}
\end{definition}
In 
a model with a single change-point at $n \gamma$, we look to estimate $\gamma$.
We use normalized versions of $\clr{j}$ and $\crl{j}$ to define an estimator
$\wgn$ of the change ratio.
\begin{definition} \label{def:crosscount}
For any sequence of $T_i^n$, using the definitions of $\clr{j}$
and $\crl{j}$ from Definition \ref{def:crossings},
define the normalized crossing processes
\begin{equation}
 \psi_{LR}(j) = \frac{\clr{j}}{n-j} - \frac{j}{n} \mbox{\;\;\; and \;\;\;}
\psi_{RL}(j) = \frac{\crl{j}}{j} - \frac{n-j}{n},\end{equation}
the maximum function
\begin{equation}
 \psi(j) = \max \left( \psi_{LR}(j), \psi_{RL}(j) \right) \end{equation}
and estimate the change-point using the 
CRECHE (CRossings Enumeration CHange Estimator) as 
\begin{equation} \wgn = \frac{1}{n} \argmin_{0 \leq j \leq n-1} 
\psi(j). \end{equation}
\end{definition}
The process $\psi_{LR}(j)$ has been designed via subtracting off
the mean of $\clr{j}$ (in a model with no change point), and 
is related to the conductance of the directed graph.

In Section \ref{sec:toy}
we prove that CRECHE $\wgn$ is $\sqrt{n}$-consistent in a related toy
model, which heuristically captures the key features of the piecewise
stationary model. We consider sampling $T_i^n$ from certain
mixtures of uniform distributions (\grmb)
and prove the following theorem:
\begin{theorem} \label{thm:main1}
For random variables $T_i^n$ generated according to \grmb\; (see
Definition \ref{def:changervs}), the estimator $\wgn$ of Definition
\ref{def:crosscount}
is $\sqrt{n}$-consistent. That is,
there exists a constant $K$, depending on $\alpha_L$, $\alpha_R$
and $\gamma$, such that for all $s$:
\begin{equation}
 \pr \left( \left| \wgn - \gamma \right| \geq \frac{s}{\sqrt{n}} 
\right) \leq \frac{K}{s^2}. \end{equation}
\end{theorem}
\begin{proof} See Appendix 
\ref{appdx:sec:fullproof}. \end{proof}

In Section \ref{sec:simulate},
we present simulation evidence that this estimator $\wgn$, applied
to \grma, performs
well in situations where the source is piecewise stationary.  
As Figure \ref{fig:markov} shows, our algorithm can even distinguish
between the output of a first order Markov chain with stationary distribution
$\mu$ and an IID process with the same distribution. Since most non-parametric
methods are based on monitoring means or densities of symbols (see
Section \ref{sec:cplit}), this illustrates a major advantage of our techniques,
since we can efficiently partition texts that a density-based method would
find indistinguishable. We hope that we could even distinguish higher
order Markov sources,  in a situation where crude bigram or trigram
counts would similarly fail (or require prohibitive amounts of data).

Our method even 
appears to give good results in
situations with a change-point between non-stationary sources --
as illustrated in Figures \ref{fig:gereng} and \ref{fig:engeng} by
 examples based on written language. This robustness to changes
in the source model should not be a surprise
since the theory of match
lengths described in Section \ref{sec:entropy} holds for a range
of independent, Markov and mixing sources.

Further, we compare the two cases where $T_i^n$ are defined according to 
\grma, as in
Definition \ref{def:matchpos}, and \grmb, as in Definition
\ref{def:changervs}. 
We 
present simulation evidence that in these two cases the functions $\psi_{LR}$
and $\psi_{RL}$ have similar behaviour, and hence the estimator $\wgn$ performs similarly
for \grma\; and \grmb.

\section{Change-point literature review} \label{sec:cplit}

The problem of detecting change-points is an important and 
well-studied one, with applications in a range of fields
listed in the book by Poor
and Hadjiliadis \cite[P1]{poor}.  For example, we mention 
bioinformatics \cite{braun},
finance \cite{aggarwal}, sensor networks \cite{nguyen}, 
climate \cite{barnett}, analysis of writing style \cite{brodsky,giron,riba}
computer security \cite{kim2} and medicine \cite{frisen}.
Our approach currently works in the case of finite alphabet sources,
and is thus naturally suited to applications in bioinformatics, computer
network intrusion detection
and analysis of writing style.

As reviewed for example in \cite{killick}, many approaches
to the change-point detection exist within a parametric
framework. The general approach is to maximise the log-likelihood, with a
penalty term that ensures the number of changes is not too large.
For example, the binary segmentation algorithm of
Scott and Knott \cite{scott} aims to detect changes in mean of
normal samples, an approach extended in work of Horv\'{a}th
\cite{horvath}
to detection of changes of mean and variance.
In general, as in \cite{killick},
it is possible to model many situations 
parametrically by supposing that between change-points, the
data is IID from a model with fixed parameter $\theta_i$,
where the parameter $\theta_i$ is itself sampled from some
prior distribution.
This parametric problem has the simplifying feature that versions of
the  likelihood ratio test can be performed, and the work 
\cite{killick} concentrates on detection of multiple change-points
in as computationally efficient a manner as possible.

In contrast non-parametric methods, required when the laws of the
random variables are not available, are less widely studied. The
book by Brodsky and Darkhovsky \cite{brodsky}
describes many such approaches, often based on detecting changes
in the mean. Other non-parametric techniques include those
based on ranks and order statistics \cite{bell}, \cite{gordon},
kernel-based methods \cite{nguyen} and approaches based on
comparing empirical distribution functions before and after
a putative change-point \cite{carlstein}, \cite{dumbgen}, 
\cite{benhariz}. The paper 
\cite{goldenshluger} extends this to consider the situation where
the source is only observed indirectly or in the presence of noise.

In particular, Ben Hariz, Wylie and Zhang \cite{benhariz} build 
on \cite{dumbgen} to produce non-parametric estimators which offer
optimal $n$-consistency (error in $\wgn$ of  $O_{\pr}(1/n)$)
under natural assumptions. However,
this approach is built on detecting changes in empirical distributions,
and so requires the stationary distributions either side of the
change-point to be different. In contrast, see Figure \ref{fig:markov},
our estimator can work well even in the case where the stationary
distributions are the same.

One further distinction to be drawn is whether
the change-point is to be detected offline
through a detailed analysis
of the data sequence, or in real-time with streaming data.
Results in the second (quickest
detection) problem are extensively reviewed in the book by Poor
and Hadjiliadis \cite{poor}.  A range of
objective and penalty functions can be considered, giving 
rise to Shiryaev's problem, Lorden's problem and others.
In essence, \cite{poor}
shows that many such problems 
 can be analysed using optimal stopping theory, and algorithms 
based on versions of Page's CUSUM test can be shown to be optimal,
as in the work of Pollak \cite{pollak} and others.
The current paper considers offline detection, but in future
work we will describe an adaptation of our match position 
approach to the quickest detection problem, using match lengths as a
proxy for log-likelihoods. 

Our approach to the problem of detection of a
change of author or language, as illustrated in Section
\ref{sec:simulate}, should be contrasted with the approach of
Gir\'{o}n, Ginebra and Riba \cite{giron,riba}. These authors
choose particular features, such as distributions
of word lengths or local frequencies of known popular words, and apply
standard change-point analysis to the resulting counts. 
A similar analysis of the homogeneity of texts is reviewed in
\cite[P169--178]{brodsky}. In
contrast, our universal approach  takes into account
all features, by finding long repeated word patterns, and 
detecting variations from uniformity in their appearance.

\section{Match lengths and entropy estimation} \label{sec:entropy}

We use calculations based on match lengths as defined by Grassberger \cite{grassberger}
and adopt the notation of Shields \cite{shields3}.
That is, we consider a string $x$ taking values in a finite alphabet $\alphb$,
which we may take to be $\{1, \ldots, |\alphb | \}$ for simplicity.
We write $x_{m}^n = (x_m, \ldots, x_n)$
for a finite subsequence.

\begin{definition} \label{def:matchlength}
For a given string $x$,
define the match length at $i$ as
\begin{equation}
 L_i^n = L_i^n(x) = 
\min \left\{ L: x_i^{i+L-1} \neq x_j^{j+L-1} \mbox{ for all } 1 \leq j \leq n, j \neq i 
\right\}. \end{equation}
\end{definition}

For a wide range of sources, it has been proved that these match lengths
can be used to consistently estimate the entropy
of data source $X$. Grassberger 
\cite{grassberger} introduced $L_i^n$, and explained heuristically
why the following result should be true:
\begin{theorem}[Shields] \label{thm:shields}
If match lengths $L_i^n$ are calculated for an  IID or mixing Markov source 
$X$ with entropy $H$,
\begin{equation} \label{eq:shieldsconv}
 \lim_{n \rightarrow \infty} \frac{\sum_{i=1}^n
 L_i^n(X)}{n \log n} = \frac{1}{H},\end{equation}
almost surely.
\end{theorem}
Theorem \ref{thm:shields} is given as 
Theorem 1 of \cite{shields}, though the proof
was completed in \cite{shields2}. Shields \cite[Section 3]{shields} shows that
(\ref{eq:shieldsconv}) does not hold in general, suggesting that determining
the class of processes for which convergence holds is a 
difficult problem. However, further progress was made by Kontoyiannis and Suhov
\cite{kontoyiannis}, who extended the convergence to the class of stationary
ergodic finite alphabet processes under a Doeblin condition. In turn,
Quas \cite{quas} extended this result to countable alphabets.

Entropy estimators given by the left-hand side of (\ref{eq:shieldsconv})
have the advantages of being 
non-parametric, computationally efficient and with fast convergence in
$n$. In particular, they
out-perform
naive plug-in estimators which estimate probability mass functions 
$p$ by empirical estimators $\hat{p}$, and then use $H(\hat{p})$
to estimate the entropy (see \cite{gao2} for a detailed simulation
analysis illustrating this).

We can heuristically understand why the result (\ref{eq:shieldsconv})
might hold, using insights given by the Asymptotic Equipartition Property for IID
sources (see \cite[Theorem 3.1.2]{cover}), or Shannon--MacMillan--Breiman theorem
for stationary ergodic
sources (see 
\cite{algoet2}). This latter result states that for 
a stationary ergodic finite alphabet source of entropy $H$, for $m$ large
enough, there exists 
a  `typical set' $\TT_m$ of strings of
length $m$ such that:
\begin{enumerate}
\item A random string lies in $\TT_m$ with probability $\geq 1 -
\epsilon$.
\item Any individual string in $\TT_m$ has probability 
$\in [2^{-m (H+\epsilon)}, 2^{-m(H-\epsilon)}]
\sim 2^{-m H}$.
\end{enumerate}
Hence, if the substring of length $m$ at point $i$ is typical, that is
$x_i^{i+m-1} \in \TT_m$, it has probability
$\sim 2^{-m H}$, so we expect to see it $\sim n 2^{-m H}$ more times.
This means that choosing $m = (\log n)/H$, we expect to see $x_i^{i+m-1}$ once more,
so match length $ L_i^n \sim (\log n)/H$.

However, it is a delicate matter to convert this intuition into a formal
proof, since there are complex dependencies between $L_i^n$ for distinct values of $i$.
The proofs of results such as Theorem \ref{thm:shields} and its
later extensions in
\cite{shields}, \cite{kontoyiannis} and \cite{quas}
 typically involve arguments involving the return times $R_k$, 
based on theorems taken from Ornstein and Weiss \cite{ornstein2,ornstein}. 
\begin{definition} \label{def:definet}
Define $R_k$ to be the time before the block $X_1^k$ is next 
seen:
\begin{equation} R_k = \min \{ t \geq 1: X_1^k = X_{t+1}^{t+k} \}.
\end{equation}
\end{definition}
It is possible to directly estimate
entropy using the return time.
Kac's Lemma \cite{kac} shows that 
 $\ep[R_k | X_1^k = x_1^k] = 1/\pr(X_1^k = x_1^k)$,
for stationary ergodic $X$.
 This intuition
was developed by
Kim \cite{kim}, who proved that $\ep[\log R_k] - k H$
converges to a constant
for independent processes and 
by Wyner (see \cite{wyner4,wyner5}), who proved asymptotic normality of
$(\log R_k - k H)/\sqrt{k}$ under the same conditions.
Corollary 2 of Kontoyiannis \cite{kontoyiannis2} extended 
this to general stationary $X$ satisfying mixing conditions. 

A simpler problem to analyse is one where the output of the source
is parsed (partitioned) into non-overlapping
blocks, and the matches take place by a blockwise
comparison (this means that `overlapping matches' are avoided). For
example, the Lempel--Ziv parsing \cite{ziv1, ziv2} breaks
the source down into consecutive blocks formed as `the shortest block
not yet seen'. In this case, as described in Cover and Thomas
\cite{cover}, a natural question with applications to many data
compression algorithms
 is to understand the asymptotic
behaviour of $L_m$, the total length of the first $m$ codewords.
 Aldous and Shields \cite{aldous} proved asymptotic normality of
$L_m$ for IID equidistributed binary processes, a result extended by 
Jacquet and Szpankowski \cite{jacquet} to IID asymmetric binary processes.

An even simpler matching was introduced by Maurer \cite{maurer}.
In this case, the output of the source is partitioned into blocks of fixed length
$\ell$, and matchings sought between them. That is,  we can define block
random variables
$Z_i = X_{(i-1)\ell+1}^{i\ell} \in \alphb^{\ell}$, and  see how
long each block takes to reappear.
\begin{definition} \label{def:defines}
For any $j$, define random variable
\begin{equation} S_j = \min \{ t \geq 1: Z_{j+t} = Z_j \}, \end{equation}
to be the return time of the $j$th block.
\end{definition}
Maurer \cite{maurer} proved that $\log S_1/\ell$ converges to the
entropy $H$ if the source is IID binary, with a similar result proved
for stationary $\psi$-mixing processes by Abadi and Galves in \cite{abadi}.
Johnson \cite{johnson18} proved a Central Limit Theorem for
the average of $\log S_i$, and hence consistency of the resulting
entropy estimates.
\section{Sources with change-points and match positions}
As described in Section \ref{sec:entropy},
previous work on match lengths has typically considered the case of a stationary
or ergodic source process; that is, one with constant distribution over
time. Next we extend this to a model with change-points.
We consider the string $x$ to be generated by the concatenation of two 
 source processes $\mu_1$ and $\mu_2$, with a sample
of length $n \gamma$ and $n (1-\gamma)$ of each. (This parameterization
is the same as that used by \cite{dumbgen} and \cite{benhariz}).
\begin{definition} \label{def:concat}
Sample two independent
infinite sequences $x(1)$, $x(2)$, where $x(i) =
x(i)_{0}^{\infty} \sim \mu_i$ for $i = 1,2$. Given length
parameter $n$ and change-point ratio $\gamma$, define the concatenated process
$x$ by
\begin{equation}
x_i = \left\{ \begin{array}{ll} x(1)_i  &	\mbox{ if $0 \leq i \leq n
\gamma - 1$, } \\
x(2)_i  &	\mbox{ if $n \gamma \leq i \leq n-1$. } \\ \end{array} \right. 
\end{equation}
\end{definition}
There has been some work concerning the properties of such a concatenated
source, though this has focussed on the case where $\gamma$ is known.
Arratia and Waterman \cite{arratia2,arratia} consider the longest common subsequence 
between the $x(1)$
and $x(2)$ process -- in contrast in some sense we consider average common
subsequences. The papers of Cai, Kulkarni and Verd\'{u} 
\cite{cai2} and of Ziv and Merhav \cite{ziv4} both consider the problem
of estimating the relative entropy from one source to another. The first
paper \cite{cai2} uses algorithms based on the Burrows-Wheeler transform
and Context Tree Weightings, the second \cite{ziv4} defines empirical 
quantities which converge to the relative entropy. However, such analysis
does not directly help us in the setting where $\gamma$ is unknown.

We now define the match positions $T_i^n$ generated by \grma:
\begin{definition} [{\bf \grma}]  \label{def:matchpos}
Taking match lengths $L_i^n$ as introduced in Definition \ref{def:matchlength},
write $\SSC_i^n$ for the positions of the match at $i$
\begin{equation}
\SSC_i^n = \left\{ j : x_i^{i+L_i^n-2} = x_j^{j+L_i^n-2}, 1 \leq j \leq n, j \neq i 
\right\} \end{equation}
and take $T_i^n$ chosen uniformly and independently at random among the elements  
of $\SSC_i^n$.
\end{definition}

Given a realisation of $x$, recall that we hope to detect the change-point -- that is,
to estimate the true value of $\gamma$.  The idea is that substrings
of $x(1)$ are likely to be similar to other substrings of $x(1)$ (and similarly 
for $x(2)$). Hence we expect that if $i \leq n \gamma-1$ 
then $T_i^n$
will tend to be $\leq n \gamma - 1$ as well. 
Similarly, for $i \geq n \gamma$, we 
expect that $T_i^n$ will tend to be $\geq n \gamma$. We consider
constructing a directed graph, with an edge between each $i$ and
the corresponding $T_i^n$, and define the crossings processes
$\clr{j}$ and $\crl{j}$ as in Definition \ref{def:crossings}.

We will look to find $j$ such that $\clr{j}$ and $\crl{j}$ are
small. However, 
consider $j=1$; then $\clr{1} = 1$, and $\crl{1}$ will be
expected to be close to 1. 
This suggests that instead of simply 
minimising $\clr{j}$ and $\crl{j}$ over $j$, we should consider a normalized
version of these quantities. The exact form of Definition \ref{def:crosscount}
is motivated by the martingale arguments used in Appendix
\ref{appdx:sec:fullproof} below.

We give theoretical and simulation results which address how
close $\wgn$ and $\gamma$ are. 
We do not expect to be able to find the change-point exactly, but hope to
prove a consistency result.
We expect that as
$n$ gets larger, the problem will get easier, though this will be controlled
by certain parameters, such as the entropy rates $H(\mu_1)$ and $H(\mu_2)$
and relative entropy rates $D( \mu_1 \| \mu_2)$ and $D(\mu_2 \| \mu_1)$.

\blah{
In this context we introduce the following notation:
\begin{definition} \label{def:ratios}
Given two stationary ergodic source processes $\mu_1$ and $\mu_2$ we 
 define
$\alpha_{12} = D(\mu_1 \| \mu_2)/H(\mu_1)$ and $\alpha_{21} = D(\mu_2 \| \mu_1)/H(\mu_2)$,
where the $D$ are relative entropy rates.
\end{definition}
These quantities play a role in the analogy between \grma\;\; for
$T_i^n$ described in Definition \ref{def:matchpos} and \grmb\;\;
described in Definition \ref{def:changervs} below. That is, by similar
arguments to those which prove the AEP, we can show that
\begin{equation} \label{eq:relentaep}
\lim_{m \tends} \frac{1}{m} \log \left( \frac{\mu_1( x_1^m)}{\mu_2( x_1^m)} \right)
\rightarrow D(\mu_1 \| \mu_2),
\end{equation}
in $\mu_1$-probability. Equivalently, $\frac{\mu_2(x_1^m)}{\mu_1(x_1^m)} \sim 
2^{-mD(\mu_1 \| \mu_2)}$.
The parameter $\alpha_L$ of Definition \ref{def:changervs} below represents
the ratio of the probabilities that a string is seen under $\mu_2$ and $\mu_1$. That is,
heuristically, for a set of `doubly-typical strings' such that 
Equation (\ref{eq:relentaep})
and the AEP of  Section \ref{sec:entropy} holds:
\begin{eqnarray} \label{eq:alphalheur}
\alpha_L & = & \frac{\mu_2(x_1^m)}{\mu_1(x_1^m)} 
\sim 2^{ -m D(\mu_1 \| \mu_2 )} 
\sim 2^{ -\log n D(\mu_1 \| \mu_2 )/H(\mu_1)} = n^{-\alpha_{12}},
\end{eqnarray}
using the heuristic from Section \ref{sec:entropy} that $m \sim \log n/H(\mu_1)$.
Similarly, $\alpha_R \sim n^{-\alpha_{21}}$.}

\section{Consistency of $\wgn$ for
toy source model} \label{sec:toy}

The theoretical analysis of $\wgn$ under \grma\; is a complex problem.
However, we prove consistency of $\wgn$ in a 
related scenario, where $T_i^n$ are generated as mixtures of
uniform distributions, which we refer to as \grmb, as follows:
\begin{definition} [{\bf \grmb}] \label{def:changervs}
Given parameters $0 <
\alpha_L < 1$ and $0 < \alpha_R < 1$, write
$\delta_L = (\gamma + (1-\gamma) \alpha_L)$ and $
\delta_R = ( \gamma \alpha_R + (1-\gamma))$. 
Define independent random variables $T_i^n$ such that:
\begin{enumerate}
\item
for each $0 \leq i \leq n \gamma - 1$,
$ \diy \pr(T_i^n = j) = \left\{  \begin{array}{ll}
\frac{1}{n \delta_L} & 0 \leq j \leq n  \gamma -1, \\
\frac{\alpha_L}{n \delta_L} & n  \gamma \leq j \leq n-1. \\
\end{array} 
\right.  $
\item
for each $n \gamma \leq i \leq n$,
$ \diy \pr(T_i^n = j) = \left\{  \begin{array}{ll}
\frac{\alpha_R}{n \delta_R} & 0 \leq j \leq n  \gamma -1, \\
\frac{1}{n \delta_R} & n  \gamma \leq j \leq n-1. \\
\end{array} 
\right.  $
\end{enumerate}
\end{definition}

Theorem \ref{thm:main1} proves that $\wgn$ is consistent in this case.
The proof of Theorem \ref{thm:main1} is built on a series of results,
and described in Appendix \ref{appdx:sec:fullproof}. 
First in
Appendix \ref{appdx:sec:matchiid}, we understand the behaviour of the 
crossings processes in a situation with no change-point. This establishes
the martingale tools we will use and allows us to prove a fluid limit,
as described in for example
\cite{darling}. That is, we show that in a model with no change-point
 the normalized crossings 
process $\psi_{LR}$
is a martingale, and use Doob's submartingale inequality to control
the deviation of the crossing process from its mean.

In Appendix \ref{appdx:sec:change}, we consider models with a change-point.
We develop the previous argument to prove that again in this case
functions related to $\psi_{LR}$ are martingales, and hence control their
difference from their mean.
We use this to 
deduce where the crossing function will be minimised, and
complete the proof of consistency of $\wgn$.

Note that in order to prove consistency of $\wgn$, it is not enough to
control the marginal distributions of $\psi_{LR}(j)$ and $\psi_{RL}(j)$;
we need uniform control of the crossings processes.
Although our proof of Theorem \ref{thm:main1}
is based on Doob's submartingale inequality, we briefly mention that it is possible
to gain an understanding of the crossings process in terms of empirical
process theory. The link between these two methods is perhaps not a surprise,
since similar relationships have been used for example by Wellner \cite{wellner}.

Recall that, given independent $U_i \sim U[0,1]$, then writing the empirical
distribution function
$ F_{n}(x) = \frac{1}{n} \sum_{i=1}^n \I( U_i \leq x)$, and 
$D_n = \sup_x |F_n(x) - x|$, Kolmogorov \cite[Theorem 1]{kolmogorov2} 
proved that $\sqrt{n} D_n$ converges in law to the so-called Kolmogorov 
distribution. This result can be understood in the context of Donsker's
Theorem, which states that
$ \sqrt{n} (F_n(x) - x)$ converges in distribution to a Brownian
bridge $B(x)$ (see for example \cite[Theorem 3.3.1, p.110]{shorack}).
The fact that the supremum of $|B(x)|$ has the Kolmogorov distribution
can be proved using the reflection principle; see for example
\cite[Proposition 12.3.4]{dudley}. 

We can use related ideas to describe the crossings process
$\psi_{LR}$ of Definition \ref{def:crosscount}
in the sense of finite dimensional distributions, in the context of 
the model without change-points used in Appendix \ref{appdx:sec:matchiid}.

\begin{lemma} \label{lem:mainemp}
For each $0 \leq i \leq n-1$, define $T_i^{n}$ independently 
uniformly distributed on $\{ 0, \ldots, n-1 \}$. The process $\sqrt{n}
\left( \psi_{LR}(\alpha n) \right) \rightarrow \sqrt{\alpha} W(\alpha/(1-\alpha))$,
in the sense of finite dimensional distributions. In particular,
%\begin{enumerate}
%\item 
for fixed $\alpha$ the
$\diy \sqrt{n}
\psi_{LR}(\alpha n) \convd N \left(0, \frac{\alpha^2}{1-\alpha}
\right)$.
\end{lemma}

However, in order to prove consistency of $\wgn$ we require uniform
control of the crossings process, meaning that martingale tools are 
natural in this context.

\section{Simulation results} \label{sec:simulate}

We illustrate by simulation results how the function $\psi(j)$ 
of Definition \ref{def:crosscount} behaves when $T_i^n$ 
are defined by match lengths, as in \grma\; of
 Definition \ref{def:matchpos}.
Note that since $0 \leq \clr{j} \leq j$ and $0 \leq \crl{j} \leq n-j$,
we know that $- \frac{j}{n} \leq \psi_{LR}(j) \leq
\frac{j^2}{n(n-j)}$ and
$- \frac{n-j}{n} \leq \psi_{RL}(j) \leq
\frac{(n-j)^2}{n j}$. See Figure \ref{fig:envelope} for a schematic 
illustration of the envelopes of these functions.

As Figure \ref{fig:envelope} might suggest, the function $\psi(j)$ can 
take large positive values for
$j$ close to $0$ or $n$. However, since we are looking for the minimum
value of $\psi$, this does not affect the analysis.
In Figure \ref{fig:nochange} we illustrate how $\psi(j)$ behaves
in a null model with no change-point. Observe that $\psi(j)$ remains close to zero except at the
end points, where it can take large positive values, as we would hope.

\begin{figure}[!htbp]
\centering
\includegraphics[width = 2.5in, height=2.5in, viewport = 150 150 500 700]{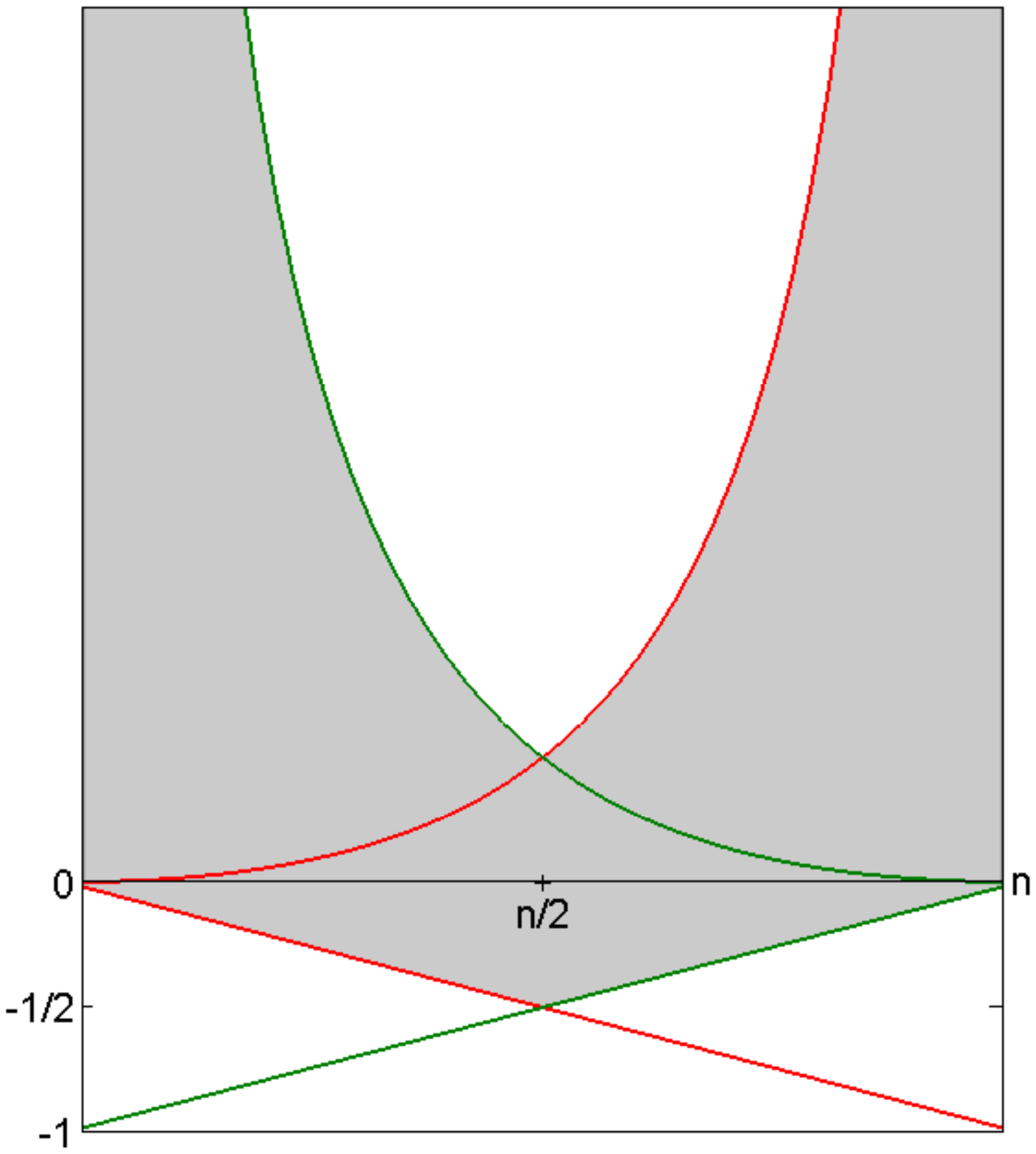}
\caption{
\label{fig:envelope} 
Schematic diagram of bounds on $\psi_{LR}$, $\psi_{RL}$ and 
$\psi$. Red curves bound values of $\psi_{LR}$, green curves
bound $\psi_{RL}$, shaded region is envelope of possible values
of $\psi$.}
\end{figure}

\begin{figure}[!htbp]
\begin{center}
\includegraphics[height=2.5in, viewport = 250 250 250 600]{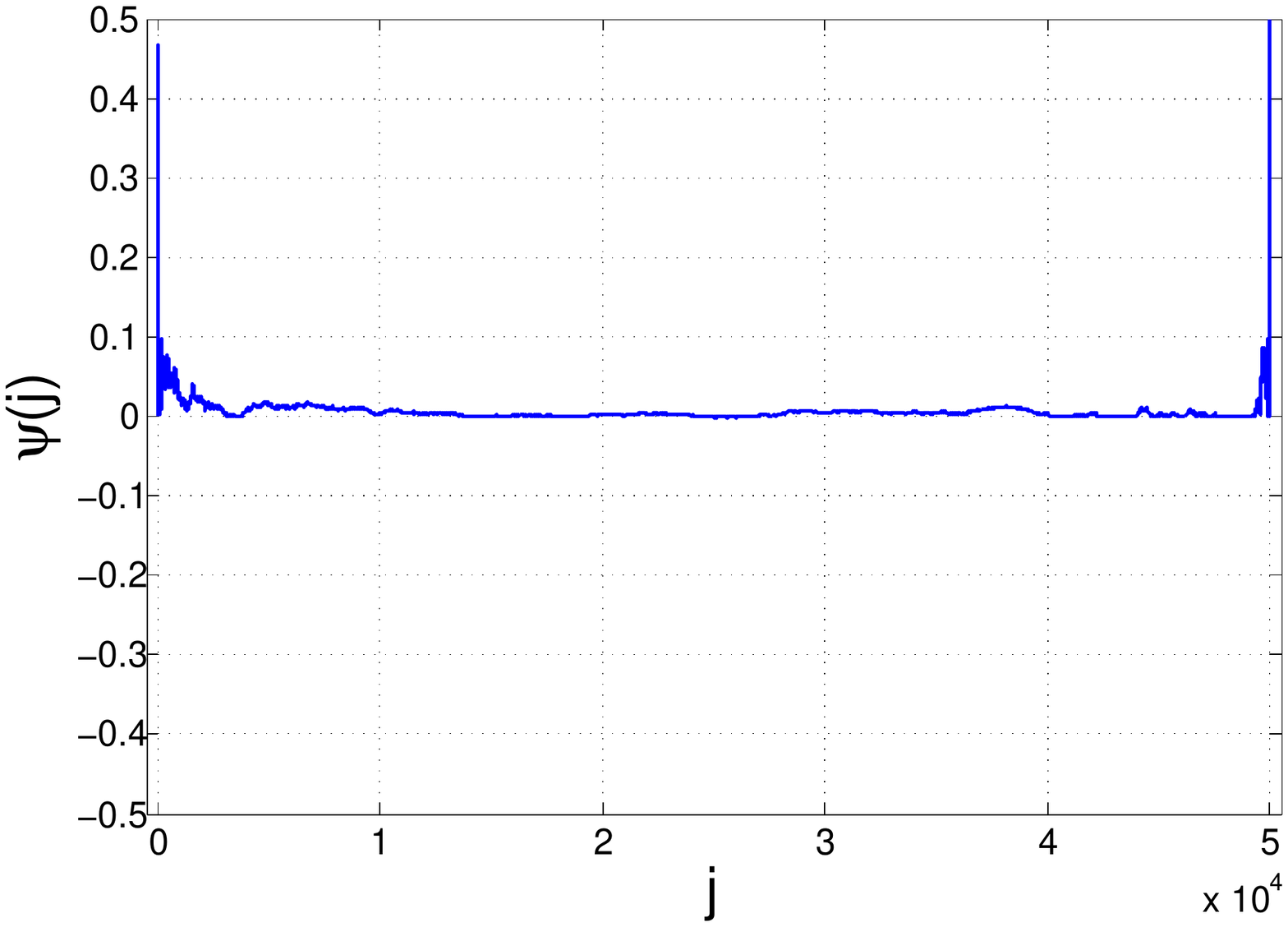}
\end{center}
\caption{\label{fig:nochange} Values of $\psi(j)$ 
simulated
from \grma\; with a source with no change-point.}
\end{figure}

In Figure \ref{fig:change} we plot values of $\psi(j)$ in a model formed by concatenating
two IID sources in the sense of Definition \ref{def:concat}. 
The change-point is marked by a vertical red line, and the function
$\psi(j)$ is minmised very close to this point, as we would hope. Further, in Figure
\ref{fig:change}, the form of the process $\psi(j)$ observed fits closely with
the theoretical properties of the corresponding process $\psi(j)$ for $T_i^n$ generated by
a toy model as in Section \ref{sec:toy}. Specifically, the function $\psi(j)$ remains close
to a piecewise smooth function, except close to the ends of the interval. Further, the piecewise
smooth function is made up of 
three components; a concave function, a linear part, and another
concave function. We explain how this pattern might be expected in Remark \ref{rem:shape}
below.

\begin{figure}[!htbp]
\begin{center}
\includegraphics[height=2.5in, viewport = 250 250 250 600]{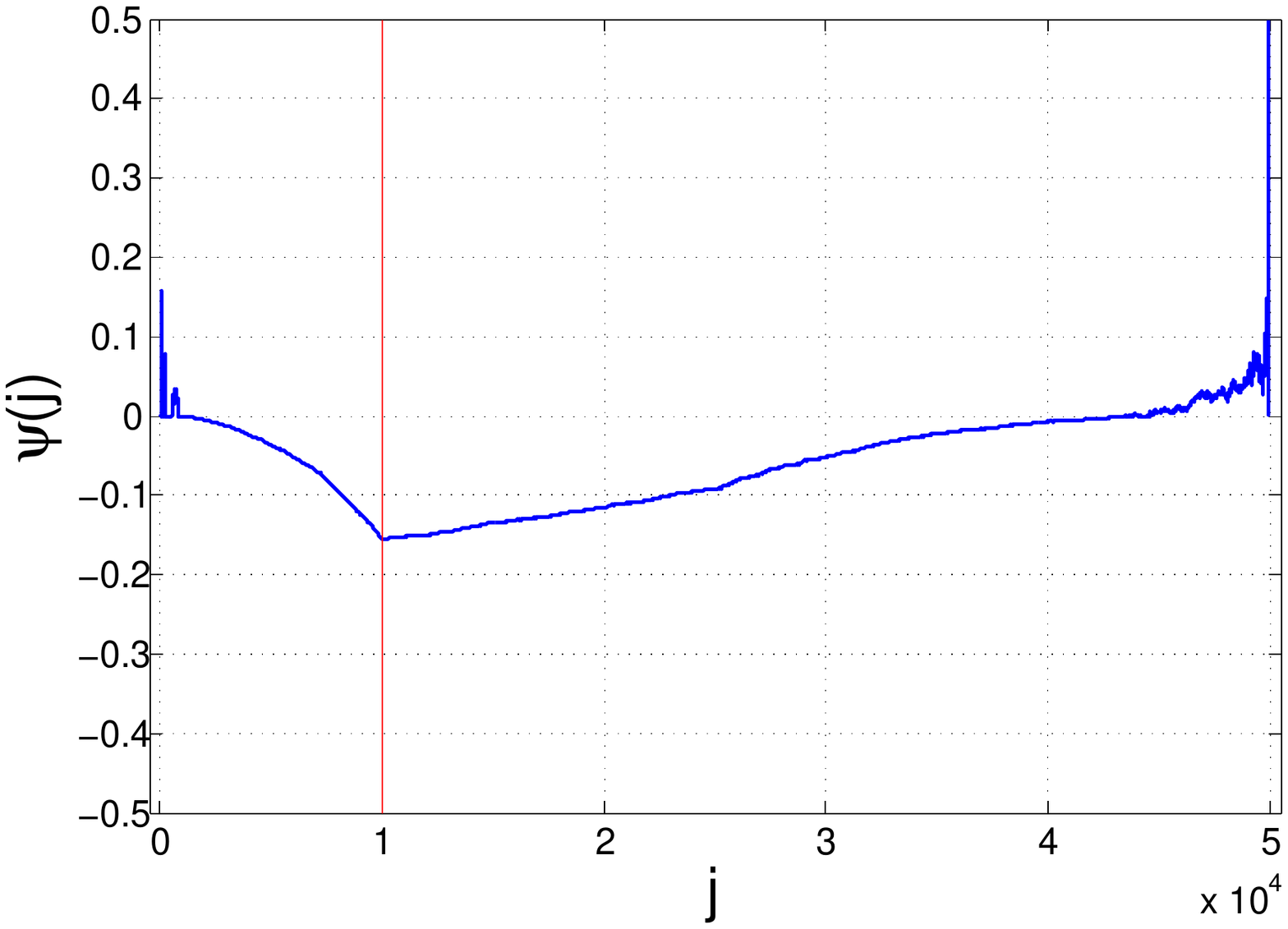}
\end{center}
\caption{\label{fig:change} Values of $\psi(j)$ simulated
from \grma\; with a source with a change-point at 
a position marked by a vertical line. The source is generated by concatenating
10,000 symbols drawn IID from the distribution $(0.1,0.3,0.6)$
with 40,000 symbols drawn IID from the distribution $(0.5,0.25,0.25)$.}
\end{figure}

We illustrate in Figure \ref{fig:markov} how the algorithm performs
over repeated trials simulated under \grma. The histogram illustrates that the algorithm
generally performs well, with a defined peak in estimates $\wgn$
close to the true value $\gamma$.  In particular,
Figure \ref{fig:markov} represents a solution to a
difficult problem, in that it shows that our algorithm can
efficiently partition a concatenation of 
a Markov chain with transition matrix $\diy \left( 
\begin{array}{lll} 
0.1 & 0.5 &  0.4 \\
0.3 &  0.4 & 0.3  \\
0.5 &  0.3 & 0.2  \end{array} 
\right) $ with stationary distribution $(0.3,0.4,0.3)$ and
an IID source with distribution 
$(0.3,0.4,0.3)$. Methods based on crude
symbol counts would fail here, but the algorithm essentially `discovers' non-uniformity
in the digram counts. The skewness of the histogram is 
perhaps to be expected, given the fact that Equations (\ref{eq:Bbound}) and
(\ref{eq:Cbound}) below are not equal (these Equations bound the performance
of the related toy \grmb).

\begin{figure}[!htbp]
\begin{center}
\begin{tabular}{cc}
\begin{minipage}{250pt}
\includegraphics[height=2.3in, viewport = 200 250 500 600]{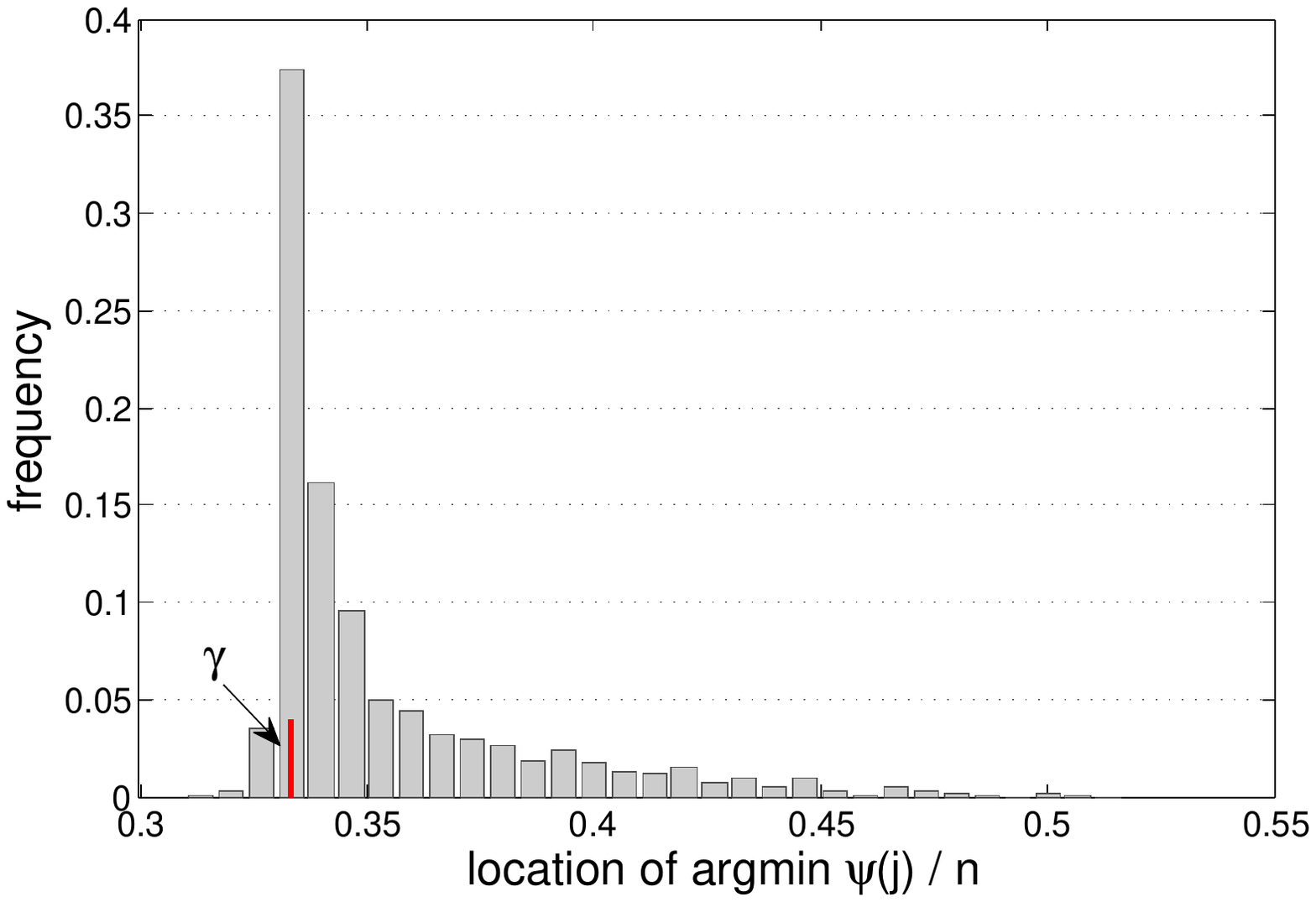} 
\end{minipage} &
\begin{minipage}{250pt}
\includegraphics[height=2.3in, viewport = 200 250 500 600]{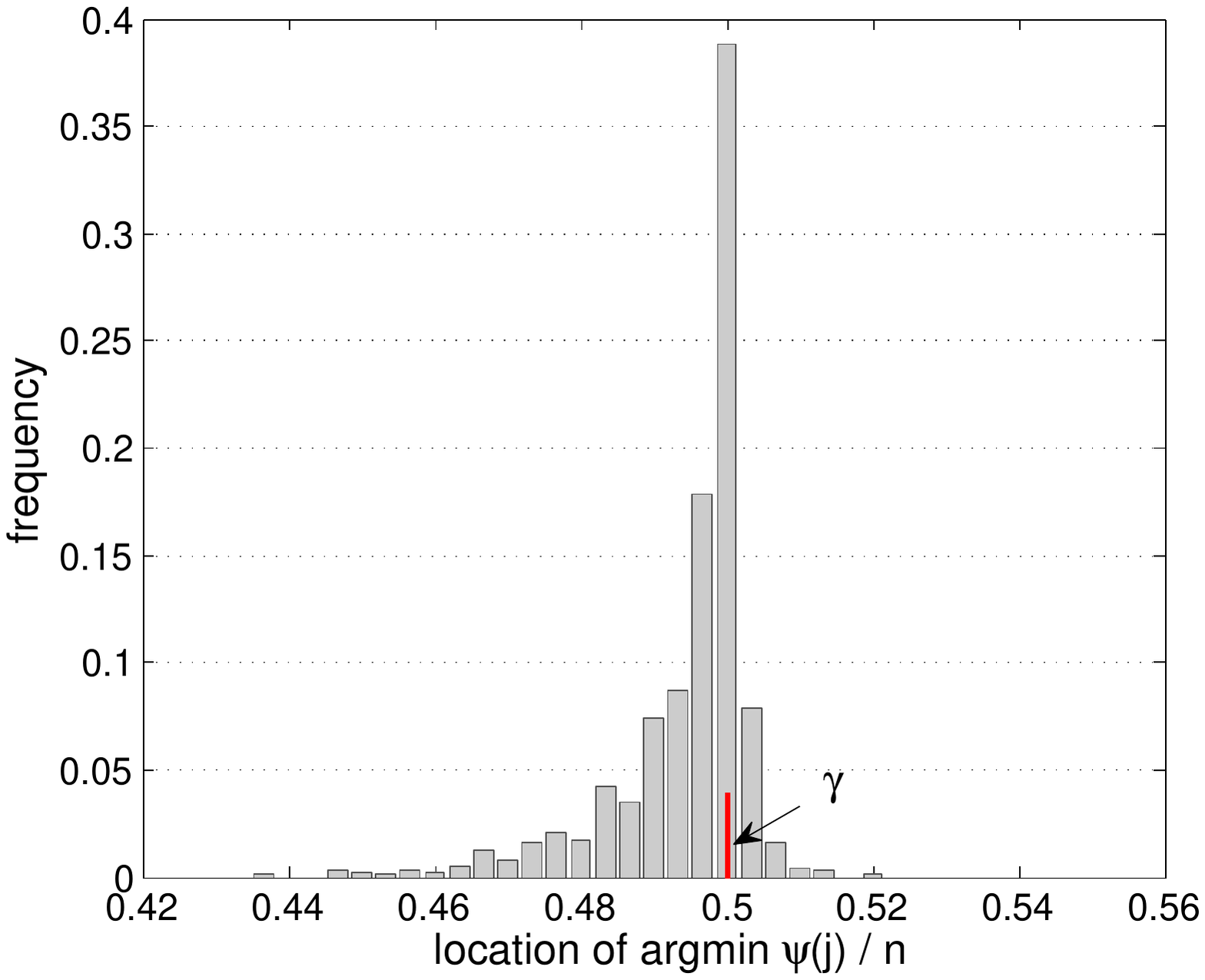} 
\end{minipage} \\
\begin{minipage}{250pt}
\includegraphics[height=2.3in, viewport = 200 250 500 600]{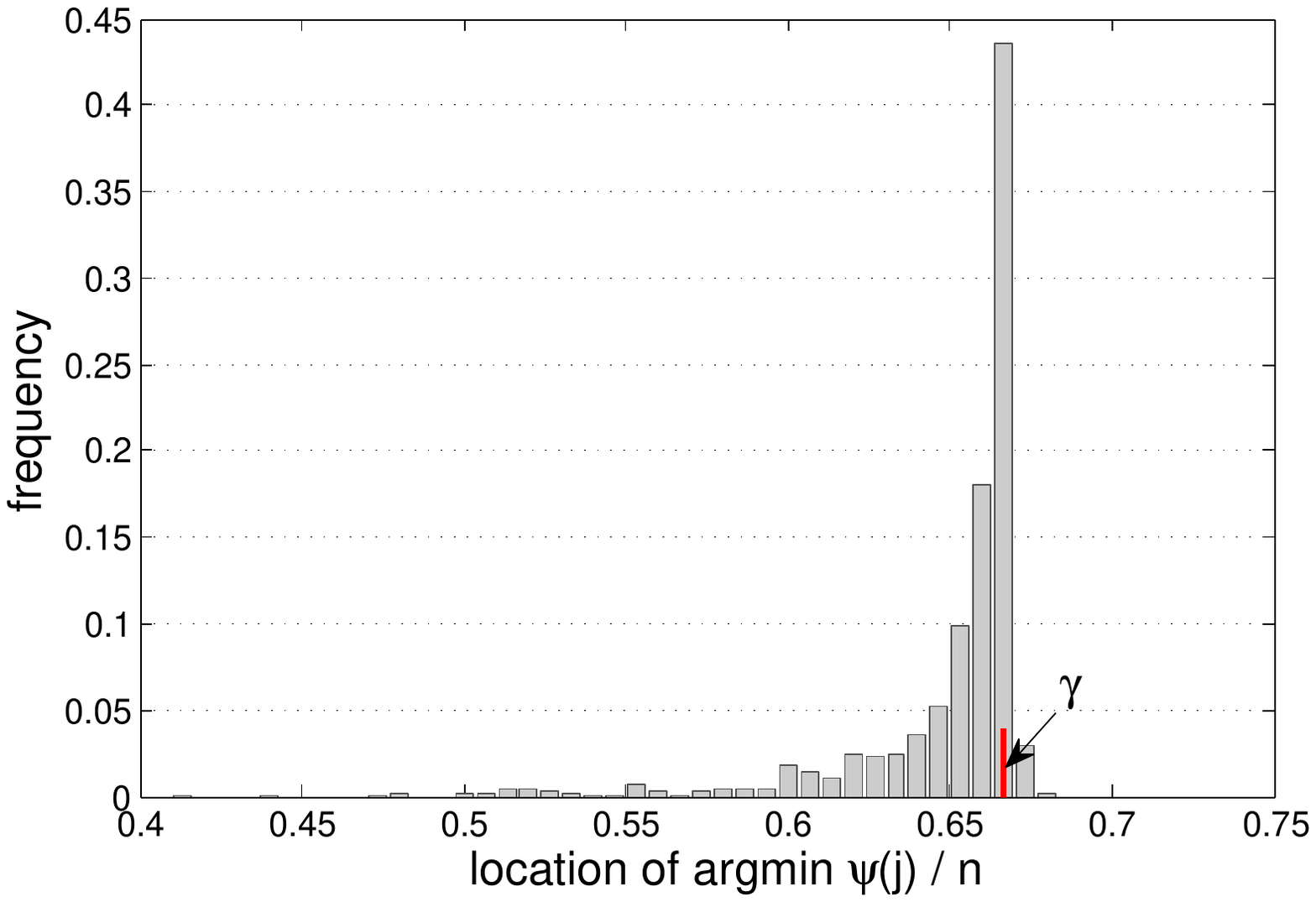}
\end{minipage} &
\begin{minipage}{250pt}
\includegraphics[height=2.3in, viewport = 200 250 500 600]{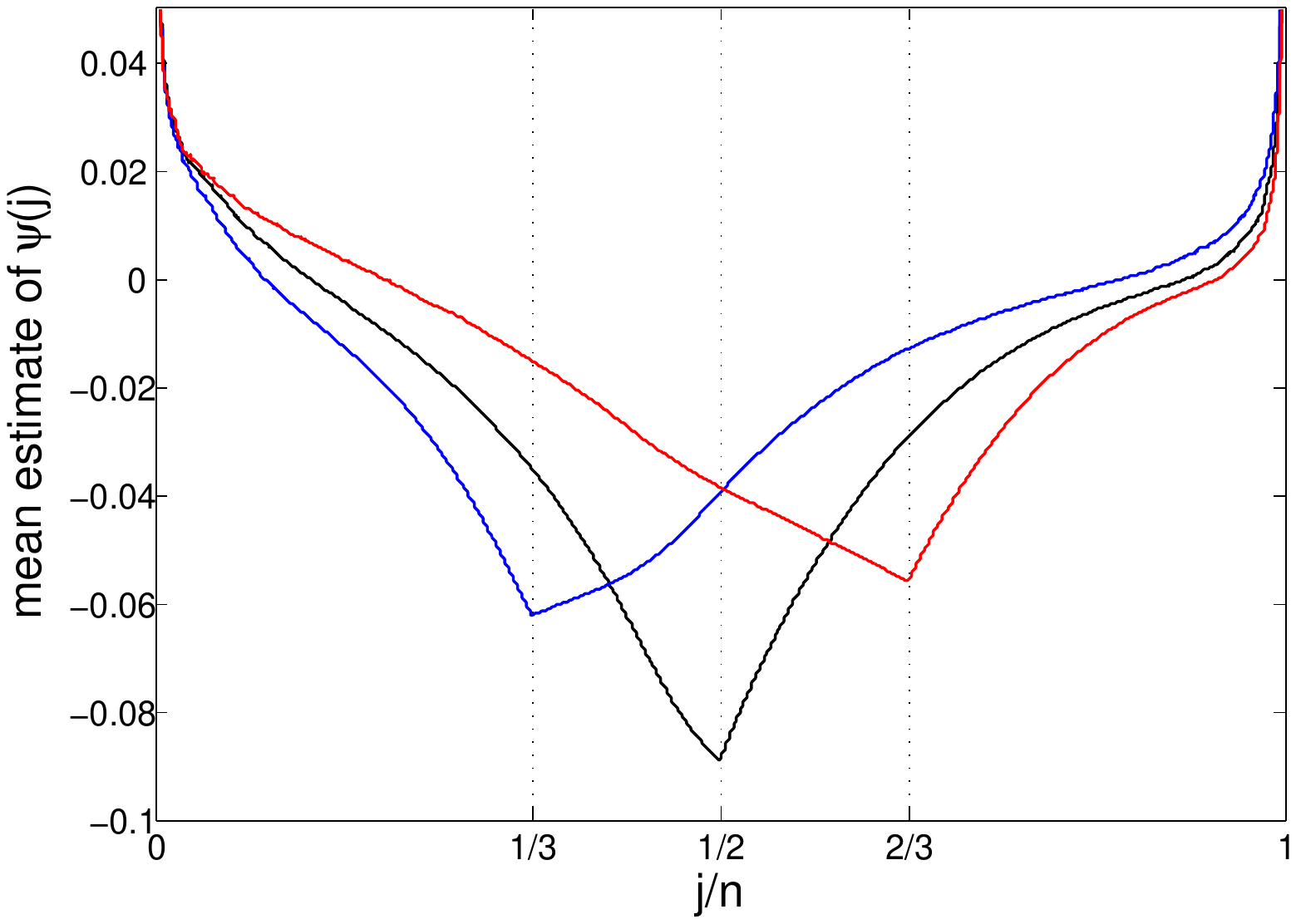}
\end{minipage}
\end{tabular}
\end{center}
\caption{\label{fig:markov} Values of $\wgn$  based on repeated trials from
 \grma\; with a source with a change-point at $\gamma$,
marked by a vertical line. In each case, we take $n=15,000$, and
the source is generated by concatenating
$n \gamma$ symbols drawn from a Markov chain with stationary
distribution $(0.3,0.4,0.3)$,
with $n(1-\gamma)$ symbols drawn IID from the distribution $(0.3,0.4,0.3)$.
The first three figures represent (a) $\gamma = 1/3$ (b) $\gamma =1/2$ (c)
$\gamma = 2/3$.  The fourth figure shows the empirical average of the curve
$\psi$ for the different values of $\gamma$.
In each case, the plot is based on 1000 trials. }
\end{figure}

Even when the two sources are not stationary, our estimator $\wgn$ appears to detect
the change-point accurately. That is, Figures \ref{fig:gereng} and \ref{fig:engeng}
illustrate that our estimator accurately detects the change-point in models built up
by concatenating natural language. In other words, in both figures, the function
$\psi(j)$ is minimised very close to the vertical line. The source of
 Figure \ref{fig:gereng} is formed by concatenating German and English versions of 
Faust, having sanitised the German text to remove  umlauts, in
order to make it look as English as possible. Figure \ref{fig:engeng} depicts
a switch between two English authors. 

Note that the value of $\psi( \wgn)$ is lower
for Figure \ref{fig:gereng} than for Figure \ref{fig:engeng}, illustrating the
natural idea that two English authors are harder to distinguish than two authors
writing in different languages. This fits with the simulation evidence provided
in \cite[Section V]{cai2}, where different languages, and 
different authors writing
in English, are distinguished by relative entropy estimates. The authors suggest
\cite[Figures 15 and 17]{cai2}
that the relative entropy from English to German and from German to English
are both around 2.5-2.6, whereas the relative entropy from one English author to
another is typically around 0.3. However, note that the paper
\cite{cai2} considers a 
different situation, in that they consider a corpus of
separate texts with authors
already distinguished, whereas this paper shows how  to partition a text by authorship.

\begin{figure}[!htbp]
\begin{center}
\includegraphics[height=2.5in, viewport = 250 250 250 700]{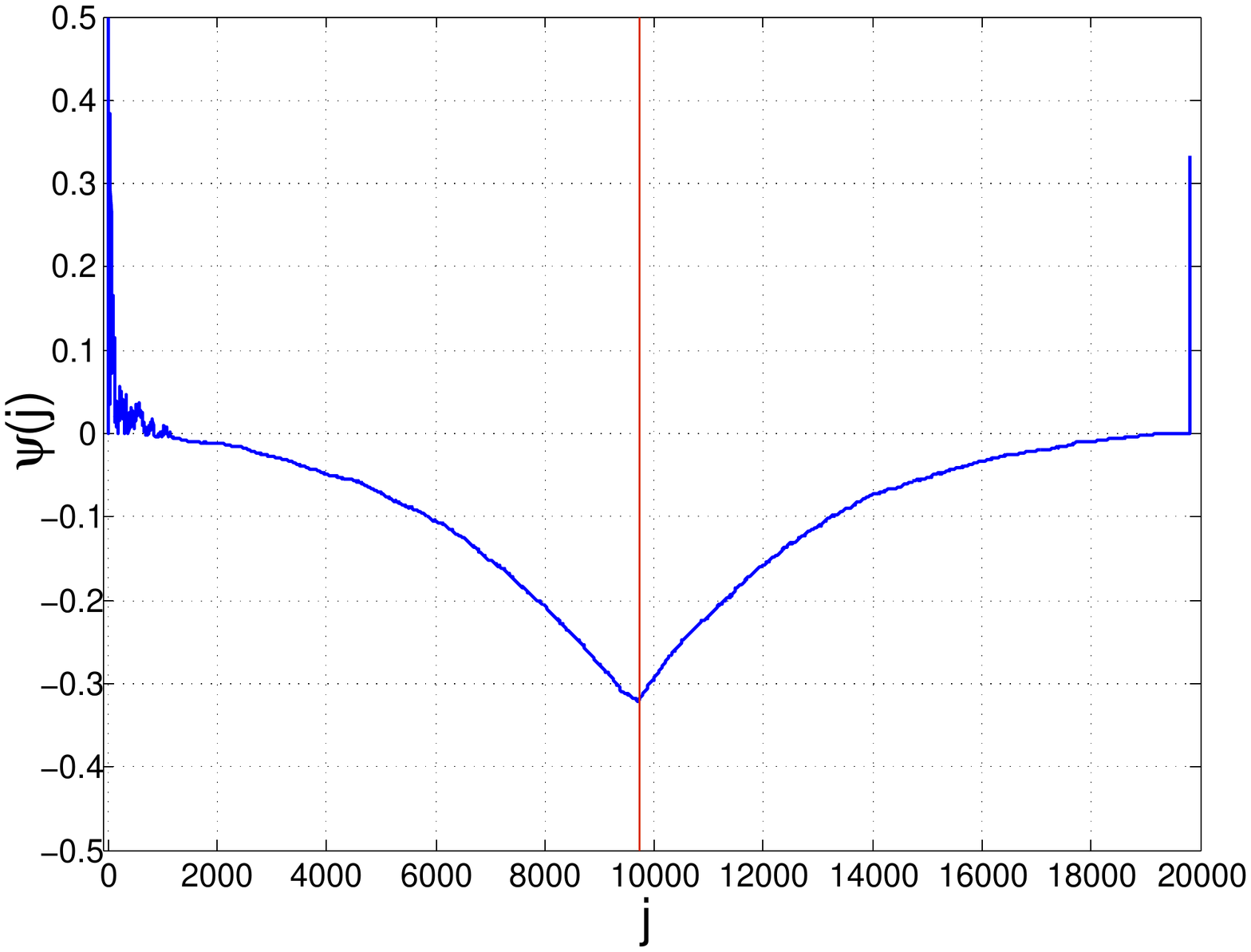}
\end{center}
\caption{\label{fig:gereng} Values of $\psi(j)$ 
generated from \grma\; with a source which switches from
German to English versions of Faust
at the position marked by a vertical line.}
\end{figure}

\begin{figure}[!htbp]
\begin{center}
\includegraphics[height=2.3in, viewport = 250 250 250 600]{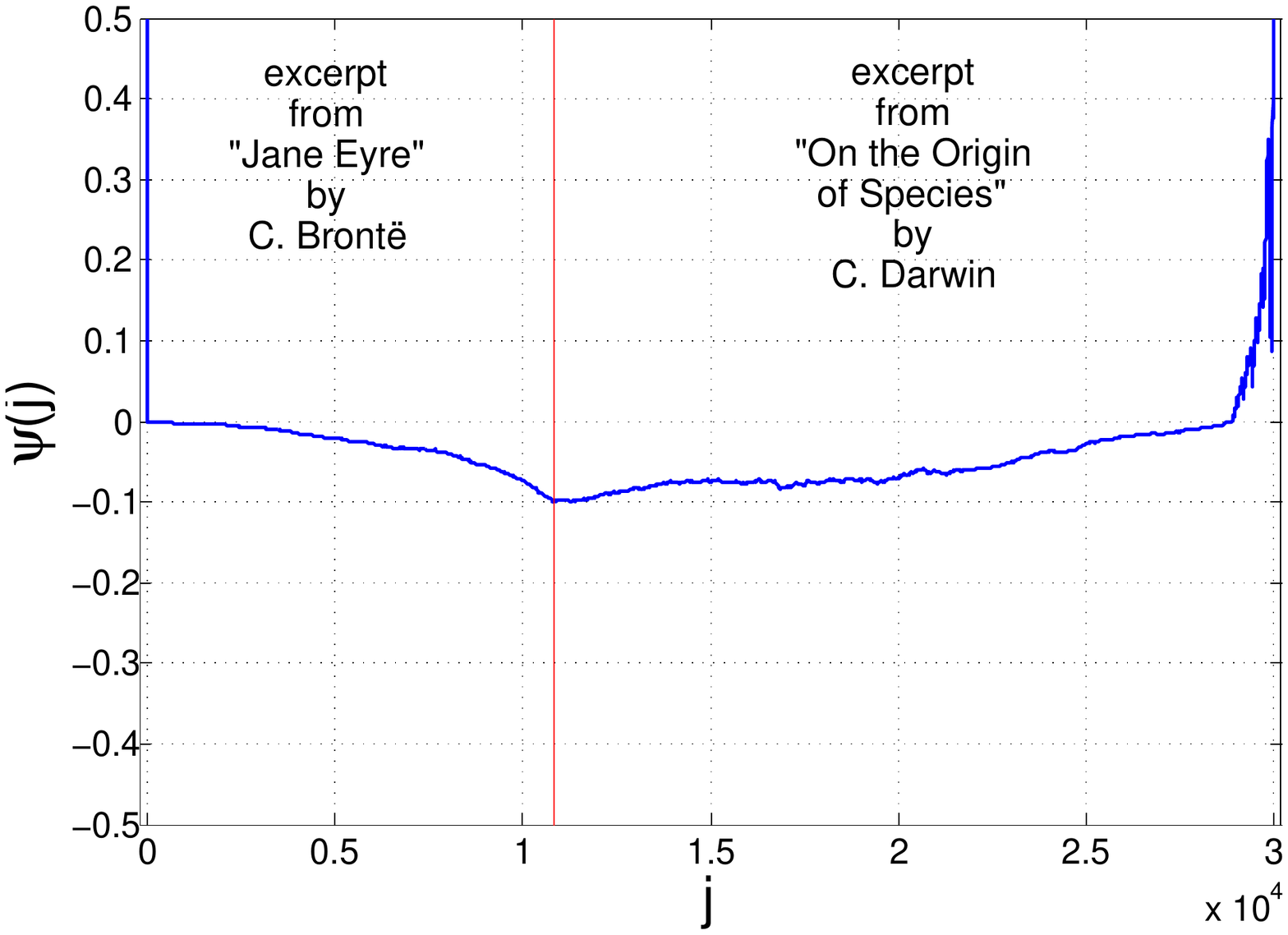}
\end{center}
\caption{\label{fig:engeng} Values of $\psi(j)$ 
generated
from \grma\; with a source
 which switches from
between English authors at the position marked by a vertical line.}
\end{figure}

\clearpage

\section{Discussion}

In this paper we have introduced a new change-point estimator, based
on ideas from information theory. We have demonstrated that it
works well for a variety of data sources, and proved 
$\sqrt{n}$-consistency in a related toy problem.
We believe that the CRECHE $\wgn$ can be adapted to detect 
change-points in a variety of related scenarios, and point out
some directions for future research.

\begin{enumerate}
\item
First, we hope to prove consistency of $\wgn$ under \grma, by
establishing a version of Theorem \ref{thm:main1}. This is likely to
require an analysis of return times similar to those described in
Section \ref{sec:entropy}, taking into account the complicated
dependencies that exist between return times of distinct and
overlapping substrings. However, we regard Theorem \ref{thm:main1}
as a significant first step towards proving such a result, since
the simulation results presented in this paper suggest that
the estimator behaves similarly in both cases. 

We note that, under
\grma,
we expect the rate of convergence of $\wgn$ to $\gamma$ to be
quicker than the $O_{\pr}(1/\sqrt{n})$ obtained in Theorem 
\ref{thm:main}, and perhaps even comparable with the $O_{\pr}(1/n)$
obtained by \cite{benhariz}. This is because a joint version of 
the Asymptotic Equipartition Property suggests that a typical
string of length $O(\log n)$ from $\mu_1$ will have $\mu_2$-probability
decaying like $O(n^{-c})$ for a certain constant. This suggests
that in terms of the toy model, we should consider crossing probabilities
$\alpha_L$ and $\alpha_R$ decaying to 0. Remark \ref{rem:faster}
below shows that in the case $\alpha_L = \alpha_R = 0$, much
faster convergence is achieved in the toy model.

\item 
Second, we believe that these consistency results should extend
to scenarios with multiple change-points (assuming the number
of change-points is low compared to the length of the data stream).
In this case, simulations show that $\psi(j)$ should have several
local minima, each corresponding to a change-point, but the 
analysis required to prove this is more involved.

\item
Third, we believe that estimators of CRECHE type can be extended
to real-valued data, as opposed to those coming from finite alphabets.
In this setting, we should be able to construct a directed graph
using closest matchings in Euclidean distance, motivated
by ideas from rate-distortion theory. We can then use
the crossings function in precisely the same way.

\item
Finally, in future work we will address the issue of quickest
detection of change-points in streaming data, in the spirit of
\cite{poor}. By estimating the typical set during the burn-in
period, we believe that match lengths can act as a proxy for
the log-likelihood in the CUSUM test.
\end{enumerate}

\appendix

\section{Proof of Theorem \ref{thm:main1}} \label{appdx:sec:fullproof}
\subsection{Matchings in an IID setting} \label{appdx:sec:matchiid}

First, we consider the behaviour of the crossings function in a simpler situation
than the \grmb\; of
Definition \ref{def:changervs}, by considering a model without a change-point,
analogous to Figure \ref{fig:nochange}. We obtain uniform control of the type required.

\begin{theorem} \label{thm:main}
For each $0 \leq i \leq n-1$ define $T_i^n$ independently 
uniformly distributed on $\{0, \ldots, n-1 \}$. For the
normalized crossings process $\psi_{LR}(j)$
of Definition \ref{def:crosscount},
for any $0 \leq \alpha \leq 1$ and $s > 0$, 
\begin{equation}
 \pr \left( \sup_{0 \leq j \leq n(1-\alpha)} \left|
\psi_{LR}(j) \right| \geq \frac{s}{\sqrt{n}}
\right)
\leq \frac{(1-\alpha)^2}{\alpha s^2}, \end{equation}
that is, $\left\{ | \psi_{LR}(j) | \leq \frac{1-\alpha}{\sqrt{\alpha n 
\epsilon}}, 0 \leq j \leq
(1-\alpha)n \right\}$ is
a pathwise $(1-\epsilon)$ confidence region on the process.
\end{theorem}

The control of  $\left| \psi_{LR}(j) \right|$ provided by
Theorem \ref{thm:main} is of optimal order, in the following two senses:

\begin{remark} \label{rem:envelope} \mbox{ }
\begin{enumerate}
\item We cannot improve the order (in $n$) of the uniform bound. By Lemma
\ref{lem:mainemp}, the
$\sqrt{n} \psi_{LR}(n(1-\alpha)) 
\convd N(0, (1-\alpha)^2/\alpha)$, so
that
\begin{eqnarray}
\liminf_{n \tends}
\pr \left( \sup_{0 \leq j \leq n(1-\alpha)} \left|
\psi_{LR}(j) \right| \geq \frac{s}{\sqrt{n}}
\right)
& \geq & \liminf_{n \tends} \pr \left( \left| \psi_{LR}(n(1-\alpha))
 \right| \geq \frac{s}{\sqrt{n}}
\right) \nonumber \\
& = & 2\left( 1 - \Phi \left( \frac{s \sqrt{\alpha}}{1-\alpha} \right) 
\right). 
\end{eqnarray}

\item
We cannot expect to control $\psi_{LR}(j)$ 
uniformly in all $j \leq n-1$ to the same
order of accuracy, as the widening envelope in Figure \ref{fig:envelope} might
suggest. Specifically, since $\clr{n-1} \sim \bino{n-1}{1/n} \convd \pois{1}$,
for any $\delta < 1$,
\begin{equation}
 \liminf_{n \tends} \pr \left( \sup_{0 \leq j \leq n-1} \left|
\psi_{LR}(j) \right| \geq \delta
\right) \geq \liminf_{n \tends} \pr(\clr{n-1} = 0) = e^{-1}. \end{equation}
\end{enumerate}
\end{remark}
Remark \ref{rem:envelope}
helps to explain the large fluctuations in $\psi(j)$ seen in Figure 
\ref{fig:nochange}. In this toy model with no change-point: for $j \leq n(1-\alpha)$, the
maximal fluctuations of $\psi_{LR}(j)$ are $O_{\pr}(1/\sqrt{n})$, but for $j \leq n$, the
maximal fluctuations are $O_{\pr}(1)$. Similarly, fluctuations in $\psi_{RL}(j)$ will be $O_{\pr}(1/\sqrt{n})$ for $j$ bounded away from zero, and
$O_{\pr}(1)$ overall.

We first prove a technical lemma regarding the thinning operation
introduced by R\'{e}nyi \cite{renyi4}. That is, for each random
variable $Y$, the $\alpha$-thinned version 
$\thinning{\alpha} Y = \sum_{i=1}^Y B_i^{(\alpha)}$,
where $B_i^{(\alpha)}$ are Bernoulli($\alpha$), independent of each other and of $Y$.
This allows us
to describe a process with binomial marginals which will prove useful
for us. In the language of \cite{alzaid} this process is a (non-stationary)
 first-order
integer-valued autoregressive $INAR(1)$ process, a discrete equivalent
of an AR(1) time series process.
\begin{lemma} \label{lem:martprop}
For fixed $N$ and $\beta$, define a process $(Y_j)$ by
$Y_0 = 0$, and recursively taking
\begin{equation}
 Y_{j+1} \sim \thinning{\frac{N-j-1}{N-j}} Y_j + U_j, \end{equation}
where $U_j \sim \bern{ \frac{\beta(N-j-1)}{N}}$ independently of all
other random variables.
Then, 
\begin{enumerate}
\item For all $j$, the $Y_j \sim \bino{j}{\beta (N-j)/N}$.
\item The process
$ \diy Z_{j} = \frac{Y_j}{N-j} - \frac{
\beta j}{N}$ is a martingale.
\item For any $d$, the process 
$\diy W_j = \left.
\left( 1 + \frac{d}{N-j} \right)^{Y_j} \right/ \left( 1 + \frac{d \beta}{N}
\right)^j$ is a martingale.
\end{enumerate}
\end{lemma}
\begin{proof} \mbox{ } 
\begin{enumerate}
\item Note that this result is true by
definition for $j=0$, we will prove it by induction in general.
Recall that for any $\alpha$, $n$ and $p$, if $Y \sim \bino{n}{p}$ 
then $\thinning{\alpha} Y  \sim \bino{n}{\alpha p}$.  Assuming $Y_j \sim \bino{j}{\beta (N-j)/N}$
for a particular $j$, then
\begin{eqnarray*}
 Y_{j+1} & \sim & \thinning{\frac{N-j-1}{N-j}}  \bino{j}{ \frac{ \beta(N-j)}{N} }
+ \bern{ \frac{ \beta(N-j-1)}{N} } \\
& \sim & \bino{j}{\frac{ \beta(N-j-1)}{N}}
+ \bern{ \frac{ \beta(N-j-1)}{N}} \\
& \sim & \bino{j+1}{\frac{ \beta(N-j-1)}{N}}.
\end{eqnarray*}
\item
This means that $\ep Y_j = \mu_j:= \beta j(N-j)/N$ for all $j$.
As a result, since
$$ \ep \left[ \left. Y_{j+1} \right| Y_j = m \right] 
=  m \frac{N-j-1}{N-j} + \frac{ \beta(N-j-1)}{N},$$
and since $Z_j = u$ exactly when $Y_j = \mu_j + u(N-j)$:
\begin{eqnarray*}
\ep [Z_{j+1} | Z_j = u] & = & \ep \left[ \left. \frac{Y_{j+1}}{N-j-1}
\right| Y_j = \mu_j + u (N-j) \right] - \frac{\beta(j+1)}{N} \\
& = & \left( \frac{\mu_j + u(N-j)}{N-j} + \frac{\beta}{N} \right) 
- \frac{\beta(j+1)}{N} \\
& = & u,
\end{eqnarray*}
by substituting for $\mu_j$.
\item Write $\alpha_j = (N-j-1)/(N-j)$, $\beta_j = \beta (N-j-1)/N$,
 $\gamma_j = 1 + d/(N-j)$ and $L
= (1 + d \beta/N)$. By a similar argument, 
since $\gamma_j^{Y_j} = u$ when $Y_j = \log u/\log \gamma_j =m$ say,
we know that
\begin{eqnarray*}
\ep \left[ \left. \gamma_{j+1}^{Y_{j+1}} \right| \gamma_j^{Y_j} = u \right]
& = & \ep \left[ \left. \gamma_{j+1}^{Y_{j+1}} \right| Y_j = m \right] \\
& = & \sum_{n=0}^m \binom{m}{n} \alpha_j^n (1-\alpha_j)^m \gamma_{j+1}^n (
\beta_j \gamma_{j+1} + 1-\beta_j) \\
& = & \gamma_j^{m} L =  u L,
\end{eqnarray*}
since $\alpha_j \gamma_{j+1} + 1- \alpha_j = \gamma_j$ and
$\beta_j \gamma_{j+1} + 1- \beta_j = L$.
\end{enumerate}
\end{proof}

\begin{proof}[Proof of Theorem \ref{thm:main}]
The key is to observe that for $T$ uniform on $\{ 0, \ldots, n-1 \}$, 
$\pr(T = j | T \geq j) = \pr(T=j)/\pr(T \geq j) = 1/(n-j)$.
This means that the LR crossing process
$\clr{j}$ is a Markov (birth and death) process. If
we know that $\clr{j} = m$, then the $m$ links that cross $j$
will cross $j+1$ independently with probability $1 - 1/(n-j)$.
In addition, there will be a contribution due to $T_j$.

In other words, the process $\clr{j}$ is distributed exactly as $Y_j$ 
in Lemma \ref{lem:martprop}, with $N=n$ and $\beta=1$. This means that by
Lemma \ref{lem:martprop},  
$ \diy \psi_{LR}(j) = \frac{\clr{j}}{n-j} - \frac{j}{n}$
is a martingale.
By a standard argument (see for example \cite[Section 14.6]{williams}), 
since $\psi_{LR}(j)$ is a martingale, Jensen's
inequality implies that $\psi_{LR}(j)^2$ is a submartingale. 
Doob's submartingale inequality \cite[Section 14.6]{williams}
states that for any non-negative submartingale
$V_j$, for any $k$ and $C$:
\begin{equation} \label{eq:doob}
 \pr \left( \sup_{1 \leq j \leq k} V_j \geq C \right)
\leq \frac{ \ep V_k}{C}.\end{equation}

Since $\clr{j} \sim \bino{j}{(n-j)/n}$, the   
$\ep \psi_{LR}(j)^2 = \var \psi_{LR}(j) 
= \var \clr{j}/(n-j)^2 = j^2/n^2(n-j)$,
 so we know that
$\ep \psi_{LR}(n(1-\alpha))^2 = (1-\alpha)^2/( \alpha n)$.

Hence, taking $V_j = \psi_{LR}(j)^2$,
$C = s^2/n$ and $k = n(1-\alpha)$ in 
Equation (\ref{eq:doob}), the theorem follows. \end{proof}

\subsection{Matching in a change-point setting} \label{appdx:sec:change}

We now use the insights of Appendix \ref{appdx:sec:matchiid} to control the 
behaviour of the crossings process $\psi_{LR}(j)$ for
\grmb, where a change-point is present at $n \gamma$.
 First we use Lemma \ref{lem:martprop} to deduce that:
\begin{proposition} \label{prop:zmartrep}
The process $\ZZZ{LR}{j}$ defined by
\begin{eqnarray}
 \ZZZ{LR}{j} & = & \left\{ \begin{array}{ll}
\left( \frac{n-j}{n \delta_L -j} \right) \left(
\psi_{LR}(j) - \mean{LR}{1}(j) \right) & \mbox{ for $0 \leq j
\leq n \gamma -1$,} \\
\left(
\psi_{LR}(j) - \mean{LR}{2}(j) \right) & \mbox{ for $n
\gamma-1 \leq j \leq n -1$,} \\
\end{array} 
\right.  \end{eqnarray}
is a martingale. Here mean functions
\begin{eqnarray} 
\mean{LR}{1}(j) & = &
- \frac{j^2}{n(n-j)} \left( \frac{(1-\gamma)(1-\alpha_L)}
{\delta_L} \right),  \label{eq:prez1} \\
\mean{LR}{2}(j) & = & 
 \left( \frac{\gamma \alpha_L}{\delta_L}
- \frac{\gamma}{\delta_R} + \frac{j}{n} \left( \frac{\gamma(1-
\alpha_R)}{\delta_R} 
\right) \right). \label{eq:prez2}
\end{eqnarray}
Further $\var \ZZZ{LR}{j}$ equals
\begin{eqnarray} 
 \frac{j^2}{n^2 \delta_L^2 (n \delta_L - j)} & & 
\mbox{ for $0 \leq j \leq n \gamma -1$,} \label{eq:zzlrvar1} \\
 \frac{ \alpha_L \gamma ( \alpha_L j + \gamma(1-\alpha_L) 
n)}{ \delta_L^2 n (n-j)}
+ \frac{ (j - \gamma n)(j - 
(1-\alpha_R) \gamma n)}{\delta_R^2 n^2(n-j)}
 &  & \mbox{ for $n \gamma-1 \leq j \leq n -1$.}  \hspace*{0.6cm}
\label{eq:zzlrvar2} 
\end{eqnarray}
\end{proposition}

\begin{proof} The key is to observe that, under \grmb,
for $k \leq n \gamma-1$:
\begin{equation}
 \pr(T_k^n = l | T_k^n \geq l) = \left\{  \begin{array}{ll}
\frac{1}{n  \delta_L-l} & \mbox{ for $0 \leq l \leq n \gamma -1$,} \\
\frac{1}{n -l} & \mbox{ for $n \gamma \leq l \leq n-1$,} \\
\end{array} 
\right.  \end{equation}
and for $k \geq n \gamma$, the $\pr(T_k^n = l | T_k^n \geq l) = 1/(n-l)$
for $l \geq n \gamma$. This means that
\begin{enumerate}
\item For $0 \leq j \leq n \gamma-1$, the
$\diy \clr{j+1} \sim \thinning{\frac{n  \delta_L-j-1}{n  \delta_L-j}} 
\clr{j} 
+ \bern{\frac{n  \delta_L-j-1}{n  \delta_L}}$.
We deduce that $\ZZZ{LR}{j}$ is a martingale in this range
and that $\clr{j} \sim \bino{j}{\frac{n  \delta_L-j}{n  \delta_L}}$
by applying Lemma \ref{lem:martprop} with $N=n  \delta_L$ and $\beta=1$.
We deduce the variance of $\ZZZ{LR}{j}$ since
$\var \ZZZ{LR}{j} = \frac{1}{(n \delta_L - j)^2} \var \clr{j}$.

\item For $n \gamma \leq j \leq n-1$,
we 
divide $\clr{j} = \clra{j}{1} + \clra{j}{2}$, where $\clra{j}{1} = \# \{ k < 
\min(j,n  \gamma): T_k \geq j \}$ and $\clra{j}{2} = \# \{ n  \gamma \leq k < j: T_k \geq j \}$. As before
\begin{enumerate}
\item $\diy \clra{j+1}{1}  \sim \thinning{\frac{n-j-1}{n-j}} \clra{j}{1}$.
In this case, since
$$ \ep [\clra{j+1}{1} | \clra{j}{1} = m] = \frac{m (n-j-1)}{(n-j)},$$
we can divide by $n-j-1$ to deduce that 
$ \diy \clra{j}{1}/(n-j)$ is a martingale.
Further, $\clra{j}{1} \sim \bino{n \gamma}{\frac{(n-j) \alpha_L}{n  \delta_L}}$.

\item $\diy \clra{j+1}{2}  \sim \thinning{\frac{n-j-1}{n-j}} \clra{j}{2} + \bern{ 
\frac{n-j-1}{n  \delta_R}}$.
 In this case, by considering $Y_s = \clra{n \gamma + s}{2}$ (since
if $j = s+n \gamma$ then $n-j = n(1-\gamma)-s$) we can write
$ \diy Y_{s+1} \sim \thinning{\frac{n(1-\gamma) - s-1}{n(1-\gamma)-s}} 
Y_s + \bern{ \frac{n(1-\gamma)-s-1}{n  \delta_R}}$. This
means we can
apply Lemma \ref{lem:martprop} with $N=n(1-\gamma)$ and $\beta = (1-\gamma)/  \delta_R$,
to deduce that 
$ \diy \frac{Y_s}{n(1-\gamma)-s} - 
\frac{s}{n  \delta_R} = \frac{ \clra{j}{2}}{n-j} - \frac{j-n \gamma}
{n \delta_R}$ is a martingale. As before
$\clra{j}{2} \sim \bino{j-n \gamma}{\frac{n-j}{n  \delta_R}}$.
\end{enumerate}

The fact that $\ZZZ{LR}{j}$ is a martingale
follows since the sum of two independent martingales is 
a martingale.
We deduce the mean and variance of $\ZZZ{LR}{j}$ since
$\diy
\var (\ZZZ{LR}{j} ) = \frac{1}{(n-j)^2} \left(
\var \clra{j}{1} + \var \clra{j}{2} \right)$.
\end{enumerate}
\end{proof}

Using this martingale characterization, and Doob's submartingale
inequality Equation
(\ref{eq:doob}), we can control $Z_{LR}$ uniformly, as before.
This allows us to control $\psi_{LR}$, as illustrated in Figure 
\ref{fig:envelopechange}.
Essentially, the confidence regions for $\psi_{LR}(j)$
are tilted versions
of the confidence region of Theorem \ref{thm:main}. This means
that the $\psi_{LR}(j)$ stay close to their mean functions for $j \leq n(1-\epsilon)$,
so that the minimum of $\psi_{LR}(j)$ must be close to the minimum of the mean
functions, namely $n \gamma$. This is illustrated in Figure \ref{fig:envelopechange}.
\begin{figure}[!htbp]
\centering
 \includegraphics[height=5in, viewport = 250 100 250 750]{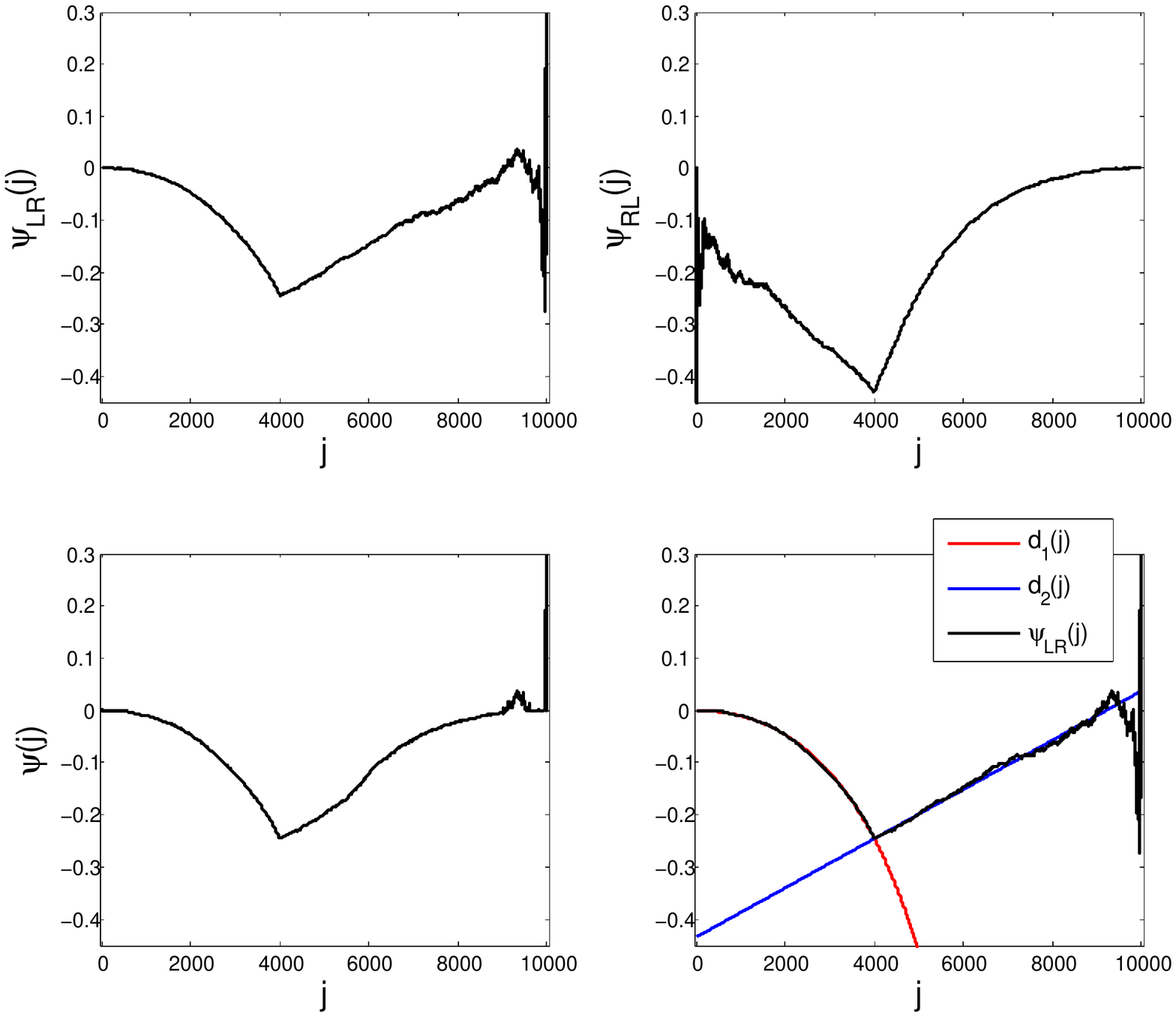}
\caption{
\label{fig:envelopechange} 
Values of (a)
$\psi_{LR}(j)$ (b) $\psi_{RL}(j)$ and (c) $\psi(j) =
\max(\psi_{LR}(j), \psi_{RL}(j))$. Data is generated
under \grmb,
with a change-point at $n \gamma=4000$. In this example,
$n=10000$, $\alpha_L = \alpha_R = 0.2$ and $\gamma = 2/5$. 
The function $\psi_{LR}(j)$
stays close to the mean functions $\mean{LR}{1}$ and $\mean{LR}{2}$  except  when
$j \geq 0.9n$, as shown in (d).}
\end{figure}

\begin{remark} 
By symmetry, the process $\ZZZ{RL}{j}$ defined by
\begin{eqnarray}
 \ZZZ{RL}{j} & = & \left\{ \begin{array}{ll}
 \left(
\psi_{RL}(j) - \mean{RL}{1}(j) \right) & \mbox{ for $0 \leq j
\leq n \gamma -1$,} \\
\left( \frac{j}{j -n \gamma(1-\alpha_L)} \right)\left(
\psi_{RL}(j) - \mean{RL}{2}(j) \right) & \mbox{ for $n
\gamma-1 \leq j \leq n -1$,} \\
\end{array} 
\right.  \end{eqnarray}
is a time-reversed martingale. Here we write
\begin{eqnarray}
\mean{RL}{1}(j) 
& = & \left( \frac{(1-\gamma) \alpha_R}{\delta_R}
- \frac{1-\gamma}{\delta_L} + \frac{(n-j)}{n} \left( \frac{(1-\gamma)(1-
\alpha_L)}{\delta_L} 
\right) \right), \label{eq:prez3} \\ 
\mean{RL}{2}(j)
 & = &  - \frac{(n-j)^2}{n j} \left( \frac{\gamma (1-\alpha_R)}
{\delta_R} \right). 
\label{eq:prez4}
\end{eqnarray}
For $n \gamma \leq j \leq n -1$ the corresponding
$\crl{j} \sim \bino{n-j}{\frac{n  \delta_R-(n-j)}{n  \delta_R}}$, the
\begin{equation} \label{eq:zzrlvar}
\var \ZZZ{RL}{j} = \frac{1}{(j-n \gamma (1-\alpha_R))^2} 
 \var \crl{j}
= \frac{(n-j)^2}{n^2 \delta_R^2 (j-n \gamma(1- \alpha_R))}.
\end{equation}
\end{remark}

\begin{remark} \label{rem:shape}
Note that the form of $\mean{LR}{i}$ and $\mean{RL}{i}$ helps explain the form of the process
$\psi(j)$ seen in Figures \ref{fig:change}
and \ref{fig:envelopechange}. That is, Equations (\ref{eq:prez1})
and (\ref{eq:prez2}) show that the mean of $\psi_{LR}(j)$ is made up of 
a concave part  left of the change-point
and a linear part right of the change-point. Similarly
by Equations (\ref{eq:prez3}) and (\ref{eq:prez4}), the mean of $\psi_{RL}(j)$
will have a linear part left of the change-point and a concave part 
right of the change-point.

In Figures \ref{fig:change} and \ref{fig:envelopechange}
we see that $\psi(j)$ remains close to
the maximum of these two curves; first the concave
$\mean{LR}{1}$ before the change-point, then the linear $\mean{LR}{2}$, followed by the concave $\mean{RL}{2}$. The exact values of 
$\gamma$, $\alpha_L$ and $\alpha_R$ will determine which curve is 
largest at a particular point.
\end{remark}

Notice that the curve $\meanone{LR}(j)$ made up of $\mean{LR}{1}(j)$ for $j \leq  n \gamma-1$
and $\mean{LR}{2}(j)$ for $j \geq n \gamma$ is minimised at $j= n\gamma$ 
with value
$\meanmin{LR} = \mean{LR}{1}(n \gamma)  = \mean{LR}{2}(n \gamma) = - \gamma^2 (1-\alpha_L)/\delta_L$.
Similarly $\meanone{RL}(j)$ is minimised at $j=n \gamma$
with value $\meanmin{RL} =
 - (1-\gamma)^2 (1-\alpha_R)/\delta_R$. 

In the proof of Theorem \ref{thm:main1} we need
to distinguish two cases, according to which of $\meanmin{LR}$ and $\meanmin{RL}$ is smaller.
We briefly remark that in the symmetric case $\alpha_L = \alpha_R$, that
$\meanmin{LR} \leq \meanmin{RL}$ if and only if $\gamma \geq 1/2$.
Further, in the limiting case $\alpha_L = \alpha_R = 0$, the two curves $\meanone{LR}$
and $\meanone{RL}$ intersect at $j = n/2$.

\begin{proof}[Proof of Theorem \ref{thm:main1}]
Without loss of generality, we will assume that $\meanmin{LR} \geq \meanmin{RL}$,
and pick $\epsilon$. Further we assume $\meanmin{LR} < 0$, which is true
if $\alpha_L < 1$. 

First, we observe that the curve $\psi$ cannot be minimised too close
 to either end of the interval of interest.
We write $\epsilon^* = - \meanmin{LR} - \epsilon$.
Recall that (see Figure \ref{fig:envelope}) $\psi(j) \geq \psi_{LR}(j) \geq -j/n$
and $\psi(j) \geq \psi_{RL}(j) \geq -(n-j)/n$. This means that
for $j < n \epsilon^*$ we know that $\psi_{LR}(j) > \meanmin{LR}+
\epsilon$, and for $j > n(1- \epsilon^*)$
we know that $\psi_{LR}(j) > \meanmin{LR}+
\epsilon$.

%We pick $\epsilon$ and 
%$\gamma \leq \epsilon^* < 1$ such that $\mean{RL}{2}(j) \geq \meanmin{LR} + 2
%\epsilon$
%for $j \geq n \epsilon^*$. Since $\mean{RL}{2}(n) = 0$, we can find
%such $(\epsilon,\epsilon^*)$ assuming $\meanmin{LR} < 0$, which is true
%if $\alpha_L < 1$.

This means that we can use the union bound and standard conditioning arguments
to decompose the error probability into three terms:
\begin{eqnarray}
\lefteqn{ \pr \left( \left| \frac{1}{n} \argmin_{j} 
\psi(j) - \gamma \right| \geq \frac{s}{\sqrt{n}} \right) } \nonumber \\
& \leq & \pr( \psi(n \gamma) > \meanmin{LR} + \epsilon)
+ \pr \left( \left. \min_{j: |j - n \gamma| \geq s \sqrt{n}} \psi(j) \leq \psi(n \gamma) 
\right| \psi(n \gamma) \leq \meanmin{LR} + \epsilon \right) \nonumber \\
& \leq & \pr( \psi(n \gamma) > \meanmin{LR} + \epsilon) \label{eq:A} \\
&  & + \pr \left( \min_{n \epsilon^* \leq j \leq n \gamma- s \sqrt{n}} 
\psi_{LR}(j) \leq \meanmin{LR} + \epsilon \right) \label{eq:B} \\
& &  + \pr \left( \min_{n \gamma + s \sqrt{n} \leq j \leq n(1- \epsilon^*)} 
\psi_{LR}(j) \leq \meanmin{LR} + \epsilon \right), \label{eq:C} 
%  & + \pr \left( \min_{n \epsilon^* \leq j \leq n} 
%\psi_{RL}(j) \leq \meanmin{LR} + \epsilon \right) \label{eq:D} 
\end{eqnarray}
using the fact that $\psi(j) = \max( \psi_{LR}(j), \psi_{RL}(j))$.
We can bound each of these terms in order.
\begin{enumerate}
\item  Observe that by the union bound 
and the form of the mean functions in
 Equations (\ref{eq:prez2}) and (\ref{eq:prez4}), we can bound
(\ref{eq:A}) by
\begin{eqnarray}
\pr( \psi(n \gamma) > \meanmin{LR} + \epsilon) 
& \leq & \pr( \psi_{LR}(n \gamma) > \meanmin{LR} + \epsilon)
+ \pr( \psi_{RL}(n \gamma) > \meanmin{RL} + \epsilon) \nonumber \\
& = & \pr( \ZZZ{LR}{n \gamma}  > \epsilon)
+ \pr( \alpha_R \ZZZ{RL}{n \gamma} > \epsilon) \nonumber \\
& = & \frac{\gamma^2 \alpha_L}
{\delta^2_L (1-\gamma) n \epsilon^2} +
\frac{(1-\gamma)^2 \alpha_R}
{\delta^2_R \gamma n \epsilon^2} \nonumber \\
& \leq & \frac{1}{n \epsilon^2} \left( \frac{\alpha_L}{1-\gamma} + 
\frac{\alpha_R}{\gamma} \right). \label{eq:Abound}
\end{eqnarray}
since by Equation (\ref{eq:zzlrvar1}) the
$\var(\ZZZ{LR}{n \gamma})   = \frac{\gamma^2 \alpha_L}
{\delta^2_L (1-\gamma) n}$, and by Equation
(\ref{eq:zzrlvar}) the
$\var(\ZZZ{RL}{n \gamma}) = \frac{(1-\gamma)^2}{n \gamma \alpha_R \delta_R}$.

\item To bound (\ref{eq:B}),
the key is to observe that the mean term $\mean{LR}{1}$ 
defined in Equation (\ref{eq:prez1})
is a concave function. This means that
for $t \geq 0$ we know that
\begin{eqnarray}
\mean{LR}{1}(n \gamma - t) - \meanmin{LR} & \geq & -
\frac{t \mean{LR}{1}(n \gamma)}{n \gamma}  = \frac{t \gamma(1-\alpha_L)}{n \delta_L},
 \label{eq:slope1} 
\end{eqnarray}
As defined in Proposition \ref{prop:zmartrep}, $\psi_{LR}(j) - \meanmin{LR}$ 
is a multiple of $\ZZZ{LR}{j}$ with a coefficient which  
 decreases in $j$, so
for $n \epsilon^* \leq j \leq n \gamma$, we can bound it by
$ \diy \frac{\gamma \alpha_L + \delta_L \epsilon}{\gamma(1-\gamma)}  \geq
\frac{n  \delta_L - j}{n-j} \geq \alpha_L.$
This means that by Equations (\ref{eq:prez1}) and (\ref{eq:slope1})
\begin{eqnarray}
\lefteqn{ \pr \left( \min_{n \epsilon^* \leq j \leq n \gamma- s \sqrt{n}} 
\psi_{LR}(j) \leq \meanmin{LR} + \epsilon \right) } \nonumber \\
& \leq & \pr \left( 
\left( \mean{LR}{1}(n \gamma - s \sqrt{n}) - \meanmin{LR} \right)
- \delta_L \left(
\sup_{n \epsilon^* \leq j \leq n \gamma - s \sqrt{n}} |\ZZZ{LR}{j}| \right)
\leq \epsilon \right) \nonumber \\
& = & \pr \left( \frac{s \gamma (1-\alpha_L)}{\delta_L \sqrt{n}} - 
\epsilon \leq  \frac{\gamma \alpha_L + \delta_L \epsilon}{\gamma(1-\gamma)} \left(
\sup_{0 \leq j \leq n \gamma} |\ZZZ{LR}{j}| \right) \right) \nonumber \\
& \leq & \left( \frac{\gamma \alpha_L + \delta_L \epsilon}{\gamma(1-\gamma)} 
\right)^2 \frac{ \var( \ZZZ{LR}{n \gamma} )}{
\left( \frac{s \gamma (1-\alpha_L)}{\delta_L \sqrt{n}} - 
\epsilon \right)^2 } \nonumber \\
& = & \frac{ \left( \gamma \alpha_L + \delta_L \epsilon \right)^2}
{\alpha_L  (1-\gamma)^3
\left( s \gamma (1-\alpha_L) - \epsilon \delta_L \sqrt{n}   \right)^2 },
\label{eq:Bbound}
\end{eqnarray}
by Doob's inequality (\ref{eq:doob}) and the variance 
expression (\ref{eq:zzlrvar1}). 

\item
Similarly, using
Equation (\ref{eq:prez2}), we know that
\begin{equation}
\mean{LR}{2}(n \gamma + t) - \meanmin{LR}  = 
\frac{t \gamma(1-\alpha_R)}
{n \delta_R}, \label{eq:slope2}
\end{equation}
meaning that
\begin{eqnarray}
\lefteqn{ \pr \left( \min_{n \gamma + s \sqrt{n} \leq j \leq n(1- \epsilon^*)} 
\psi_{LR}(j) \leq \meanmin{LR} + \epsilon \right) } \nonumber \\
& \leq & \pr \left( 
\left( \mean{LR}{1}(n \gamma + s \sqrt{n}) - \meanmin{LR} \right)
- \left(
\sup_{n \gamma + s \sqrt{n} \leq j \leq n (1-\epsilon^*)} |\ZZZ{LR}{j}| \right)
\leq \epsilon \right) \nonumber \\
& \leq & \frac{ \var( \ZZZ{LR}{n (1-\epsilon^*)})}{
\left(  \frac{s \gamma (1-\alpha_R)}{\delta_R \sqrt{n}}
- \epsilon \right)^2} \nonumber \\
& = & 
\frac{\gamma + \alpha_L}{ \epsilon^* \delta_L
\left(  \frac{s \gamma (1-\alpha_R)}{\delta_R }
- \epsilon \sqrt{n} \right)^2}, \label{eq:Cbound}
\end{eqnarray}
since (\ref{eq:zzlrvar2}) implies that
\begin{eqnarray*}
\lefteqn{
\var \left( \ZZZ{LR}{n (1-\epsilon^*) } \right) } \\
& = & \frac{1}{n \epsilon^*} \left( \frac{\alpha_L \gamma 
(\alpha_L (1-\epsilon^*) + \gamma(1-\alpha_L))}
{\delta_L^2} + 
\frac{(1-\epsilon^* - \gamma)( 1-\epsilon^* - \gamma (1-\alpha_R))}{\delta_R^2}
\right) \\
& \leq & \frac{1}{n \epsilon^*} \left( \frac{\alpha_L \gamma}{\delta_L}  +1
\right) = \frac{ \gamma + \alpha_L}{n \epsilon^* \delta_L}
\end{eqnarray*}
%\item
%We can control (\ref{eq:D}) by observing that 
%$ \diy \alpha_R \leq \frac{ j - n \gamma(1-\alpha_R)}{j} \leq \delta_R$, so that
%\begin{eqnarray*}
%\lefteqn{ \pr \left( \min_{n \epsilon^* \leq j \leq n} 
%\psi_{RL}(j) \leq \meanmin{LR} + \epsilon \right) } \\
%& \leq & \pr \left( \min_{n \epsilon^* \leq j \leq n} (\mean{RL}{2}(j) - \meanmin{LR}
%-\epsilon) + \min_{n \epsilon^* \leq j \leq n} 
%\frac{ j - n \gamma(1-\alpha_R)}{j} \ZZZ{RL}{j} \leq 0 \right) \\
%& \leq & \pr \left( \delta_R \sup_{n \epsilon^* \leq j \leq n} 
%\left| \ZZZ{RL}{j} \right| \geq \epsilon \right) \\
%& = & \frac{\delta_R^2 \var \ZZZ{RL}{n \epsilon^*}}{\epsilon^2}
%= \frac{(1-\epsilon^*)^2 }{n \epsilon^2 (\epsilon^* - \gamma(1-\alpha_R))^2}
%\end{eqnarray*}
%since $(\mean{RL}{2}(j) - \meanmin{LR}
%-\epsilon) \geq \epsilon$.
\end{enumerate}
The result follows on adding together the contributions from
Equations (\ref{eq:Abound}), (\ref{eq:Bbound}) and (\ref{eq:Cbound}). We can
choose for example $\epsilon = \gamma^3 (1-\alpha_L) s/(\delta_L \sqrt{n})$,
since $s/\sqrt{n} \leq (1-\gamma)$,  since the assumption that 
$\meanmin{LR} \geq \meanmin{RL}$ ensures that 
$\epsilon \leq \gamma(1-\gamma)^2 (1-\alpha_R) s/\delta_R \sqrt{n}$.
Putting these terms together, we deduce that we can take 
\begin{eqnarray}
K & = &
\left( \frac{\alpha_L}{1-\gamma} + \frac{\alpha_R}{1-\gamma} \right)
\frac{\delta_L^2}{\gamma^6 (1-\alpha_L)^2}
+ \frac{  (\alpha_L + \gamma^2(1-\alpha_L)(1-\gamma))^2}
{\alpha_L (1-\gamma^2)^2(1-\gamma)^3(1-\alpha_L)^2}  \nonumber \\
 & & + \frac{ (\gamma+ \alpha_L)}{\gamma^2(1-\alpha_L)(1-\gamma(1-\gamma))}
\frac{\delta_R^2}{(\gamma^2(1-\alpha_R)^2 (1-(1-\gamma)^2)^2}.
\label{eq:kbound} \end{eqnarray}
\end{proof}

\begin{remark} \label{rem:faster}
Note that the form of (\ref{eq:kbound}) suggests that 
as $\alpha_L$ tends to zero, then $K$ will tend to infinity, meaning
that this is the hardest case. Of course,
the case $\alpha_L = \alpha_R = 0$ will have no crossings of $n \gamma$,
so should be the easiest case. 
We can indeed do much better
by adapting the argument slightly.  Without loss of generality assume that $\gamma \leq 1/2$,
and recall that
in this case $\meanmin{LR} = -\gamma$, 
and we can choose $\epsilon = 0$, so that $\epsilon^* = \gamma$.
This means that Equations (\ref{eq:A}) and (\ref{eq:B}) are zero,
since $\var \ZZZ{LR}{n \gamma} = 0$, and since the interval $[n \epsilon^*, n \gamma-s
\sqrt{n}]$ is empty.
Then taking $\alpha_L = \alpha_R$ in Equation (\ref{eq:Cbound}) gives
$ \diy \frac{ (1-2 \gamma)^2}{s \gamma^3}$. Overall, this means that
$$ \pr \left( |\wgn - \gamma| \geq \frac{s}{\sqrt{n}} \right)
\leq \frac{ (1-2 \gamma)^2}{s \gamma^3},$$
suggesting that the estimator is $\sqrt{n}$-consistent in this case. 

In fact, we can do better. Since the
interval $[n \epsilon^*, n \gamma-1]$ is empty,
we can strengthen the bound on (\ref{eq:B}) to
deduce that $\pr \left( \min_{n \epsilon^* \leq j \leq n \gamma- 1} 
\psi_{LR}(j) \leq \meanmin{LR} + \epsilon \right)= 0$.
Further notice that when $\gamma = 1/2$, the
$\pr(\wgn \neq \gamma) = 0$, since the interval  
$[n \gamma + 1 \leq j \leq n (1-\epsilon^*)]$ is again empty.

Otherwise, we divide the interval into further subintervals, using
a similar argument to that used to obtain (\ref{eq:Cbound}). 
Since Equation (\ref{eq:zzlrvar2}) gives 
$\diy \var \ZZZ{LR}{b} = \frac{(b-\gamma n)^2}{(1-\gamma)^2 n^2 (n-b)}$,
for any $n \gamma \leq a \leq b$ we know that
\begin{eqnarray}
\pr \left( \min_{a \leq j \leq b} 
\psi_{LR}(j) \leq \meanmin{LR}  \right)
& \leq & \pr \left( 
\left( \mean{LR}{1}(a) - \meanmin{LR} \right)
\leq 
\sup_{a \leq j \leq b} |\ZZZ{LR}{j}|  \right) \nonumber \\
& = &  \pr \left( 
\frac{(a-\gamma n) \gamma}{n(1-\gamma)}
\leq 
\sup_{a \leq j \leq b} |\ZZZ{LR}{j}| \right) \nonumber \\
& \leq & \left( \frac{ n(1-\gamma)}{(a-\gamma n) \gamma} \right)^2 
\var( \ZZZ{LR}{b} ) \nonumber \\
& = & \frac{ (b - \gamma n)^2}{(n-b) \gamma^2 (a - \gamma n)^2}.
\label{eq:tounion}
\end{eqnarray}
This means that we can pick a constant $C > 1$, and divide the interval
$[n \gamma + 1, n(1-\gamma)]$ into subintervals
$[a_k,b_k]$, where $a_k = n \gamma + C^k$ and $b_k = \min \left( n \gamma + C^{k+1},
n(1-\gamma) \right)$,
where $k = 0, \ldots, K-1$, with $K = \log( n(1-2 \gamma))/\log C$. 
Applying the
union bound to these intervals, we deduce by Equation (\ref{eq:tounion})  that 
\begin{equation}
 \pr(\wgn \neq \gamma) \leq \frac{C^2}{\gamma^2} \frac{K}{n},\end{equation}
or in other words that the probability that the estimator makes a mistake
is $O((\log n)/n)$.
Up to the factor of $\log n$, this probability is of optimal order, since 
for $\gamma < 1/2$ independence implies that
\begin{eqnarray*}
\lefteqn{ \liminf_{n \rightarrow \infty} n \pr( \wgn \neq \gamma) } \\
& \geq & \liminf_{n \rightarrow \infty} n
 \pr \left( \left\{ \psi_{LR}(n \gamma + 1) \leq \meanmin{LR} \right\}
\bigcap \left\{ \psi_{RL}(n \gamma + 1) \leq \meanmin{LR} \right\} \right) \\
& \geq & \liminf_{n \rightarrow \infty} n \pr( \clr{n \gamma + 1} = 0 ) \pr ( 
\crl{n \gamma+1} = 0) \\
& = & \frac{e^{-1}}{(1-\gamma)},
\end{eqnarray*}
as $\clr{n \gamma + 1} \sim \bern{ \frac{n(1-\gamma) - 1}{n(1-\gamma)}}$
and $\crl{n \gamma + 1} \sim \bino{n(1-\gamma) - 1}{\frac{1}{n(1-\gamma)} }
\convd \pois{1}$.
\end{remark}

\section*{Acknowledgements}
This work was funded by a grant from the Ministry of Defence, via
the Underpinning Defence Mathematics programme. We would like to thank
Christophe Andrieu of Bristol University and
Tim Boxer of the Industrial Mathematics KTN for their support and advice.

\bibliography{../../bibliography/papers}

\end{document}